\documentclass[a4paper]{amsart}
\usepackage[all,cmtip]{xy}
\usepackage{amsmath}
\usepackage{amssymb}
\usepackage{color}
\usepackage{float}
\usepackage{varioref}
\usepackage{appendix}
\usepackage[bookmarksnumbered,colorlinks,plainpages]{hyperref}
\usepackage{cite}

\DeclareMathOperator{\rank}{rk}
\DeclareMathOperator{\lcm}{lcm}
\renewcommand{\mod}{\operatorname{mod}}
\DeclareMathOperator{\CM}{CM}
\newcommand{\rperp}[1]{#1^\perp}
\newcommand{\taumin}{\tau^{-}}
\newcommand{\ffrac}[2]{#1/#2}
\newcommand{\sCM}{\underline{\CM}}
\newcommand{\CMgr}[2]{\CM^{#1}\textrm{-}#2}
\newcommand{\sCMgr}[2]{\underline{\CM}^{#1}\textrm{-}#2}
\newcommand{\La}{\Lambda}

\newcommand{\mmod}[1]{\mathrm{mod}\textrm{-}#1}
\newcommand{\Der}[1]{\mathrm{D}^b(#1)}
\newcommand{\Knull}[1]{\mathrm{K}_0(#1)}
\newcommand{\Dsing}[2]{\mathrm{D}_{\mathrm{Sg}}^{#1}(#2)}

\newcommand{\ra}{\rightarrow}

\newcommand{\la}{\lambda}

\newcommand{\de}{\delta}
\newcommand{\De}{\Delta}
\newcommand{\ga}{\gamma}
\newcommand{\Ga}{\Gamma}
\newcommand{\si}{\sigma}
\newcommand{\ci}{\circ}

\newcommand{\coh}[1]{\mathrm{coh}\textrm{-}#1}
\newcommand{\vect}[1]{\mathrm{vect}\textrm{-}#1}
\newcommand{\svect}[1]{\underline{\mathrm{vect}}\textrm{-}#1}
\newcommand{\vx}{\vec{x}}
\newcommand{\vc}{\vec{c}}
\newcommand{\vom}{\vec{\omega}}
\newcommand{\modgr}[2]{\mathrm{mod}^{#1}\textrm{-}#2}
\newcommand{\projgr}[2]{\mathrm{proj}^{#1}\textrm{-}#2}

\newcommand{\Sing}[2]{\mathrm{D}^{#1}_{Sg}(#2)}
\newcommand{\XX}{\mathbb{X}}
\newcommand{\PP}{\mathbb{P}}
\newcommand{\FF}{\mathbb{F}}
\newcommand{\ZZ}{\mathbb{Z}}

\newcommand{\Dd}{\mathcal{D}}

\newcommand{\Oo}{\mathcal{O}}
\newcommand{\Tt}{\mathcal{T}}
\def\AA{\mathbb{A}}
\newcommand{\DD}{\mathbb{D}}
\newcommand{\EE}{\mathbb{E}}
\newcommand{\Hom}[3]{\mathrm{Hom}_{#1}(#2,#3)}
\newcommand{\Homstab}[3]{\mathrm{\underline{Hom}}_{#1}(#2,#3)}
\newcommand{\Ext}[4]{\mathrm{Ext}^{#1}_{#2}(#3,#4)}
\newcommand{\Hh}{\mathcal{H}}

\newcommand{\eulerchar}[1]{\chi_{#1}}
\newcommand{\Ll}{\mathcal{L}}
\newcommand{\lra}{\longrightarrow}
\newcommand{\lup}[1]{\xrightarrow[]{#1}}
\newcommand{\up}[1]{\stackrel{#1}{\lra}}
\newcommand{\slope}[1]{\mu\,#1}
\DeclareMathOperator{\dual}{D}

\newcommand{\image}[1]{\mathrm{im}(#1)}
\newcommand{\kernel}[1]{\mathrm{ker}(#1)}
\newcommand{\res}[2]{{#1}_{|#2}}
\newcommand{\Aa}{\mathcal{A}}
\newcommand{\sAa}{\underline{\Aa}}

\newtheorem*{Thm-A}{Theorem A}
\newtheorem*{Thm-C}{Theorem C}
\newtheorem*{Lemma-B}{Lemma B}
\newtheorem{proposition}{Proposition}[section]
\newtheorem{theorem}[proposition]{Theorem}
\newtheorem{corollary}[proposition]{Corollary}

\newtheorem{lemma}[proposition]{Lemma}

\theoremstyle{definition}

\newtheorem{remark}[proposition]{Remark}

\numberwithin{equation}{section}
\newcommand{\LL}{\mathbb{L}} \newcommand{\set}[1]{\{#1\}}
\newcommand{\Set}[2]{\{#1 \,|\, #2\}}
\newcommand{\Pp}{\mathcal{P}}
\newcommand{\sPp}{\underline{\Pp}}
\newcommand{\Ff}{\mathcal{F}} \newcommand{\vy}{\vec{y}}
\newcommand{\bx}{\bar{x}}
\newcommand{\px}{\dot{x}}
\newcommand{\Ss}{\mathcal{S}}
\newcommand{\ulPp}{\underline{\Pp}}
\newcommand{\ulPhi}{\underline{\Phi}}
\newcommand{\coker}[1]{\mathrm{coker}{#1}}
\newcommand{\incl}{\hookrightarrow}
\newcommand{\End}[2]{\mathrm{End}_{#1}(#2)}
\newcommand{\add}[1]{\mathrm{add}(#1)} \newcommand{\iso}{\cong}

\newcommand{\nilop}[1]{\tilde{\Ss}(#1)}
\newcommand{\snilop}[1]{\underline{\tilde{\Ss}}(#1)}

\newcommand{\sMFgr}[2]{\underline{\textrm{MF}}^{#1}(#2)}
\newcommand{\bp}{\bar{p}}
\newcommand{\QQ}{\mathbb{Q}}
\DeclareMathOperator{\determ}{det}
\newcommand{\sHom}[3]{\mathrm{\underline{Hom}}_{#1}(#2,#3)}
\newcommand{\sEnd}[2]{\mathrm{\underline{End}}_{#1}(#2)}
\newcommand{\euler}[2]{\langle #1,#2\rangle}
\newcommand{\seuler}[2]{\langle\langle #1,#2\rangle\rangle}
\newcommand{\vz}{\vec{z}}

\begin{document}
\title[Nilpotent operators and weighted projective lines]{Nilpotent
  operators and \\ weighted projective lines}

\author[D. Kussin]{Dirk Kussin}
\address{Dipartimento di Informatica -- Settore di Matematica\\
Universit\`a degli Studi di Verona\\
Strada le Grazie 15 -- Ca' Vignal 2\\
37134 Verona\\
Italy}
\email{dirk@math.uni-paderborn.de}

\author[H. Lenzing]{Helmut Lenzing}
\address{Institut f\"{u}r Mathematik\\
Universit\"{a}t Paderborn\\
33095 Paderborn\\
Germany}
\email{helmut@math.uni-paderborn.de}

\author[H. Meltzer]{Hagen Meltzer}
\address{Instytut Matematyki\\
Uniwersytet Szczeci\'nski\\
70451 Szczecin\\
Poland}
\email{meltzer@wmf.univ.szczecin.pl}
\subjclass[2010]{Primary: 16G60, 18E30, 14J17. Secondary: 16G70, 47A15}
\maketitle

\begin{abstract}
  We show a surprising link between singularity theory and the
  invariant subspace problem of nilpotent operators as recently
  studied by C.~M.\ Ringel and M.\ Schmidmeier, a problem with a longstanding
  history going back to G.\ Birkhoff. The link is established via
  weighted projective lines and (stable) categories of vector bundles
  on those.

  The setup yields a new approach to attack the subspace problem. In
  particular, we deduce the main results of Ringel and Schmidmeier for
  nilpotency degree $p$ from properties of the category of vector
  bundles on the weighted projective line of weight type $(2,3,p)$,
  obtained by Serre construction from the triangle singularity
  $x^2+y^3+z^p$. For $p=6$ the Ringel-Schmidmeier classification is
  thus covered by the classification of vector bundles for tubular
  type $(2,3,6)$, and then is closely related to  Atiyah's
  classification of vector bundles on a smooth elliptic curve.

  Returning to the general case, we establish that the stable
  categories associated to vector bundles or invariant subspaces of
  nilpotent operators may be naturally identified as triangulated
  categories. They satisfy Serre duality and also have tilting objects
  whose endomorphism rings play a role in singularity theory. In fact,
  we thus obtain a whole sequence of triangulated (fractional)
  Calabi-Yau categories, indexed by $p$, which naturally form an
  ADE-chain.
\end{abstract}

\section{Introduction and main results}

In recent work Ringel and Schmidmeier thoroughly studied the
classification problem for invariant subspaces of nilpotent linear operators
(in a graded and ungraded version) \cite{Ringel:Schmidmeier:2006,
  Ringel:Schmidmeier:2008a, Ringel:Schmidmeier:2008b}. This problem
has a long history and can actually be traced back to Birkhoff's
problem ~\cite{Birkhoff:1934},
dealing with the classification of subgroups of finite abelian
$p$-groups. We note that Simson~\cite{Simson:2002} determined the
complexity for the classification of indecomposable objects, depending
on the nilpotency degree. Even more generally, Simson considered the
classification problem for chains of invariant subspaces, without
however attempting an explicit classification. For additional
information on the history of the problem we refer
to~\cite{Ringel:Schmidmeier:2008b,Simson:2002}.
The main achievement of \cite{Ringel:Schmidmeier:2008b}
is such an explicit classification for $p\leq6$ where the case $p=6$,
yielding tubular type, is the most difficult one and very much related
to the representation theory of so-called tubular algebras, a problem
initiated and accomplished by Ringel in~\cite{Ringel:tubular}.
In the present paper we describe an unexpected access to the invariant
subspace problem for graded nilpotent operators through the theory of
weighted projective lines.

Our approach focusses on the categorical structure and the global aspects (Serre duality, tilting, Calabi-Yau dimension) of the invariant subspace problem, and yields complete and satisfying results making use of the knowledge of the structure of vector bundles on a weighted projective line. The study complements the treatment from~\cite{Ringel:Schmidmeier:2008b}, which is in the spirit linear algebra and more explicit concerning the structure of individual representations. Our study further links the problem with other
mathematical subjects (singularities, vector bundles, Cohen-Macaulay
modules, Calabi-Yau categories) and largely enhances our knowledge
about the original problem. In particular, our treatment yields a uniform treatment of three formerly unrelated problems, each forming a socalled ADE-chain: the study of triangle singularities of type $(2,3,p)$,  the invariant subspace problem of linear operators which are nilpotent of degree $p$, and finally the representation theory of an equioriented quiver of Dynkin type $\AA_{2(p-1)}$ equipped with all nilpotency relations of degree $3$.

Let $\XX=\XX(2,3,p)$ denote the weighted projective line of weight
type $(2,3,p)$, where the integer $p$ is at least
$2$. Following~\cite{Geigle:Lenzing:1987}, the category $\coh\XX$ of
coherent sheaves on $\XX$ is obtained by applying Serre's
construction~\cite{Serre:fac} to the (suitably graded) \emph{triangle singularity} $x_1^2+x_2^3+x_3^p$.

The properties of $\coh\XX$ are very similar to the properties of the category of (algebraically) coherent sheaves on a smooth projective curve or of the category of (analytically) coherent sheaves on a compact Riemann surface. Common to the three categories is that they are hereditary, that is, extensions $\Ext{i}{}{-}{-}$ for degree $i\geq2$ vanish. Further each coherent sheaf $X$ decomposes into a direct sum of a locally free sheaf $E$ (a vector bundle) and a sheaf having finite support. This yields the concept of \emph{rank} of $X$ defined as the rank of the locally free sheaf $E$.  Another property shared by these categories is that the Euler characteristic of the geometric object is a measure for the complexity of the category of coherent sheaves, especially for the category of vector bundles $\vect\XX$.  On the other hand, a property characteristic for weighted projective lines is the existence of a tilting object $T$ for the category of coherent sheaves on a weighted projective line~\cite{Lenzing:1997}. This means an object $T$ where all self-extensions $\Ext1{}{T}{T}$ vanish and which, moreover, generates the category by forming extensions, kernels of monomorphisms and cokernels of epimorphisms. This in turn implies a close relationship to the representation theory of the finite dimensional endomorphism algebra $A$ of $T$, formally expressed in an equivalence of the bounded derived categories $\Der{\coh\XX}$ and $\Der{\mmod{A}}$ of coherent sheaves on $\XX$ and finite dimensional modules over $A$, respectively. In the non-weighted situations such a tilting object only exists in the case of the projective line or, in the case of Riemann surfaces, for the Riemann sphere.

We recall that the Picard group of
$\XX$ is naturally isomorphic to the rank one abelian group
$\LL=\LL(2,3,p)$ on three
generators $\vx_1,\vx_2,\vx_3$ subject to the relations
$2\vx_1=3\vx_2=p\vx_3$. Up to isomorphism the line bundles
are therefore given by the system $\Ll$ of twisted structure sheaves
$\Oo(\vx)$ with $\vx\in\LL$. A key aspect of our paper is a properly
chosen subdivision of the system  $\Ll$ of all line bundles into two
disjoint classes $\Pp$ and $\Ff$ of line bundles, called
\emph{persistent} and \emph{fading}, respectively. This subdivision
arises from the partition of $\LL$ into the subsets
$\PP=\ZZ\vx_3\sqcup(\vx_2+\ZZ\vx_3)$ and $\FF=\LL\setminus \PP$, each
consisting of cosets modulo $\ZZ\vx_3$.

Let $[\Ff]$ denote
the ideal of all
morphisms in the category $\vect\XX$ of vector bundles which factor through a finite direct sum of
fading line bundles. We recall that a \emph{Frobenius category} is an
exact category (Quillen's sense) which has enough projectives and
injectives, and where the projective and the injective objects
coincide.

\begin{Thm-A} \label{thm:A}
  Assume $\XX$ has weight type $(2,3,p)$ with $p\geq2$. Then the
  following holds.
\begin{enumerate}
\item[(1)] The category $\vect\XX$ is a Frobenius category with the
  system $\Ll$ of all line bundles as the indecomposable
  projective-injective objects.
\item[(2)] The factor category $\vect\XX/[\Ff]$ is a Frobenius
  category with the system $\Pp$ of persistent line bundles forming the full subcategory $\Pp/[\Ff]=\Ll/[\Ff]$ of
  indecomposable projective-injective objects; we use the notation $\ulPp$.
\item[(3)] The stable categories $\vect\XX/[\Ll]$ and
  $(\vect\XX/[\Ff])/(\Pp/[\Ff])$ are naturally equivalent as
  triangulated categories; we use the notation $\svect\XX$.
\end{enumerate}
\end{Thm-A}

\begin{Lemma-B} \label{lemma:B}
The category $\ulPp$ is equivalent to the path category of the quiver
\begin{equation} \label{eqn:ladder}
\def\ci{\circ}
\xymatrix@R18pt@C15pt{
 \cdots\ar[r]&\ci\ar[d]_y\ar[r]^x &\ci \ar[d]_y\ar[r]^x &
 \ci\ar[d]_y\ar[r]^x
 &\cdots\ar[r]^x&\ci\ar[d]_y\ar[r]^x&\ci\ar[d]_y\ar[r]&\cdots\\
\cdots\ar[r]&\ci \ar[r]_x       &\ci \ar[r]_x       &\ci\ar[r]_x
&\cdots\ar[r]_x& \ci\ar[r]_x&\ci\ar[r]&\cdots\\}
\end{equation}
modulo the ideal given by all commutativities $xy=yx$ and all
nilpotency relations $x^p=0$.
\end{Lemma-B}
Hence the category of finite dimensional contravariant $k$-linear
representations of the
above quiver with relations is naturally isomorphic to the category
$\mod{\ulPp}$ of finitely presented right modules over $\ulPp$. We
call $\ulPp$ (and by
abuse of language sometimes also the quiver~\eqref{eqn:ladder}) the (infinite)
\emph{$p$-ladder}. Clearly a
right $\ulPp$-module is exactly a morphism $U\ra M$ between two
$\ZZ$-graded $k[x]/(x^p)$-modules, where $x$ gets degree $1$. The
category $\nilop{p}$ consists of all those morphisms which are
monomorphisms. As a full subcategory $\nilop{p}$ is extension-closed
in $\mod{\ulPp}$, hence $\nilop{p}$ inherits an exact structure which
is actually Frobenius with the projectives from $\mod{\ulPp}$ as the
projective-injective objects.   Note further that  $\ZZ$ acts on
$\nilop{p}$ by degree shift,
denoted by $s$.

\sloppy
\begin{Thm-C} \label{thm:C}
Assume $\XX$ has weight type $(2,3,p)$. Then the functor
$$
\Phi\colon\vect\XX \lra\mod{\ulPp}, \quad E\mapsto \ulPp(-,E)
$$
induces equivalences $\Phi\colon\ffrac{\vect\XX}{[\Ff]}\lra
\nilop{p}$ of Frobenius categories and
$\ulPhi\colon\svect\XX\lra\snilop{p}$ of triangulated categories,
respectively. Moreover, under the functor $\Phi$ the degree shift by
$\vx_3\in\LL$ on $\vect\XX$ corresponds to the degree shift $s$ by
$1\in\mathbb{Z}$ on $\mod\ulPp$ and its subcategory $\nilop{p}$.
\end{Thm-C}
\fussy
This theorem implies (most of) the results
from~\cite{Ringel:Schmidmeier:2008b} from results on the categories
$\vect\XX$; it further has a significant number of additional
consequences. For instance, we show that the triangulated category
$\svect\XX=\snilop{p}$, has Serre duality and admits a tilting
object. Indeed, we give explicit constructions for two tilting objects
$T$ and $T'$ in $\vect\XX$ with non-isomorphic endomorphism rings,
yielding by Theorem~C explicit tilting objects for $\snilop{p}$.
The two tilting objects have
$2(p-1)$ pairwise indecomposable summands. The  endomorphism algebra
of $T$ is the representation-finite Nakayama algebra $A(2(p-1),3)$
given by the quiver
\begin{equation}\label{eqn:nakayama}
1\up{x}2\up{x}3\up{x}4\up{x}\cdots \up{x}2p-3 \up{x} 2(p-1)
\end{equation}
with all nilpotency relations $x^3=0$. Let $[1,n]$ denote the linearly
ordered set $\set{1,2,\ldots,n}$. Then the endomorphism ring of $T'$
is given as the incidence algebra $B(2,p-1)$ of the poset
$[1,2]\times[1,p-1]$, that is the algebra given by the fully
commutative quiver
\small
\begin{equation}\label{eqn:rectangle}
\xymatrixcolsep{0.5pc}
\xymatrix@!C=36pt@R18pt{ 1\ar @{->}[r]\ar
  @{->}[d]\ar@{..}[rd] & 2\ar
  @{->}[r]\ar @{->}[d]\ar@{..}[rd] &3\ar[d]\ar[r]&  \cdots \ar
  @{->}[r] & p-2\ar @{->}[r]\ar @{->}[d]\ar@{..}[rd] & p-1\ar
  @{->}[d]\\
  1'\ar @{->}[r] & 2'\ar @{->}[r] &3'\ar[r] &\cdots \ar @{->}[r] & (p-2)'\ar
  @{->}[r] & (p-1)'. }
\end{equation}\normalsize
It is well-known that rectangular diagrams of this shape appear in
singularity theory, see for instance
\cite{Gabrielov:Dynkin,Ebeling:book}.

\sloppy
Algebraically, an established method to investigate the complexity of
a singularity is due to R.~Buchweitz~\cite{Buchweitz}, later revived
by D.~Orlov~\cite{Orlov:2009} who primarily deals with the graded
situation. Given the $\LL$-graded triangle singularity
$S=k[x_1,x_2,x_3]/(x_1^2+x_2^3+x_3^p)$ this amounts to consider the
Frobenius category $\CMgr\LL{S}$ of $\LL$-graded maximal
Cohen-Macaulay modules, its associated stable category $\sCMgr\LL{S}$
and the \emph{singularity category of $S$}
defined as the quotient
$\Dsing\LL{S}=\Der{\modgr\LL{S}}/\Der{\projgr\LL{S}}$. It is shown in
\cite{Buchweitz,Orlov:2009} that the two constructions yield naturally
equivalent triangulated categories $\sCMgr\LL{S}=\Dsing\LL{S}$. It
further follows from \cite{Geigle:Lenzing:1987} that sheafification
yields natural equivalences $\CMgr\LL{S}\up\sim\vect\XX$ with the
indecomposable projective $\LL$-graded $S$-modules corresponding to
the line bundles on $\XX$, and then inducing natural identifications
\begin{equation}\label{eq:threeCY}
\Dsing\LL{S}=\sCMgr\LL{S}=\svect\XX, \textrm{ where }\XX=\XX(2,3,p).
\end{equation}
In particular, comparing the sizes of the triangulated categories
$\Der{\coh\XX}$ and $\svect\XX$
by the ranks of their Grothendieck groups yields
$$\rank (\Knull{\svect\XX})-\rank (\Knull{\coh\XX})=p-6,$$ a formula
nicely illustrating the effects of (an $\LL$-graded version of)
Orlov's theorem~\cite{Orlov:2009}.
Moreover, for each $p$ the triangulated categories from (\ref{eq:threeCY}) are
fractional Calabi-Yau where, up to cancelation, the Calabi-Yau
dimension equals $1-2\chi_\XX$. Here $\chi_\XX=1/p-1/6$ is the
orbifold Euler characteristic of $\XX$. By Theorem~C all
these assertions transfer to properties of $\snilop{p}$. For details
and further applications we refer to Section~\ref{sect:applications}.
\fussy

The structure of the paper is as follows. In Section~\ref{sect:basics}
we recall some basic properties of weighted projective lines. In Section~\ref{sect:covers} we survey fundamental properties of projective covers and injective hulls in $\vect\XX$. There we also introduce an important class of vector bundles of rank two, called Auslander bundles. These will play a key role for the proofs of the main results given in Section~\ref{sect:proofs}.
Section~\ref{sect:applications} is devoted to applications
concerning the categories $\nilop{p}=\vect\XX/[\Ff]$ and $\snilop{p}=\svect\XX$. In Appendix~\ref{sect:appendix} we present a tilting object for the category $\svect\XX$ which is important for the applications discussed in Section~\ref{sect:applications}.

\section{Definitions and basic properties} \label{sect:basics}

We recall some basic notions and facts about weighted projective
lines. We restrict our treatment to the case of three weights.  So
let $p_1,\,p_2,\,p_3 \geq 2$ be integers, called weights.
Denote by $S$ the commutative algebra
$$S=\frac{k[X_1,X_2,X_3]}{\bigl(X_1^{p_1}+X_2^{p_2}+X_3^{p_3}\bigr)}
=k[x_1,x_2,x_3].$$
Let $\LL=\LL(p_1,p_2,p_3)$ be the abelian group
given by generators $\vx_1,\,\vx_2,\,\vx_3$ and defining relations
$p_1 \vx_1=p_2 \vx_2=p_3 \vx_3=:\vc$. The $\LL$-graded algebra $S$ is
the appropriate object to study the triangle singularity
$x_1^{p_1}+x_2^{p_2}+x_3^{p_3}$.
The element $\vc$ is called the
\emph{canonical element}. Each element $\vx\in\LL$ can be written in
canonical form
\begin{equation}
  \label{eq:normalform}
  \vx=n_1 \vx_1+ n_2 \vx_2+n_3 \vx_3+m\vc
\end{equation}
with unique $n_i$, $m\in\mathbb{Z}$, $0\leq n_i <p_i$.

The algebra $S$ is $\LL$-graded by setting $\deg x_i=\vx_i$
($i=1,2,3$), hence $S=\bigoplus_{\vx\in\LL}S_{\vx}$. By an
$\LL$-graded version of
the Serre construction~\cite{Serre:fac}, the weighted
projective line $\XX=\XX(p_1,p_2,p_3)$ of weight type $(p_1,p_2,p_3)$
is given by its category of coherent sheaves
$\coh\XX=\mod^{\LL}(S)/\mod_0^{\LL}(S)$, the quotient category of
finitely generated $\LL$-graded
modules modulo
the Serre subcategory of graded modules of finite length.  The abelian
group $\LL$ is ordered by defining the positive cone
$\{\vx\in\LL\mid\vx\geq\vec{0}\}$ to consist of the elements of the
form $n_1 \vx_1 +n_2 \vx_2 +n_3 \vx_3$, where $n_1,\,n_2,\,n_3 \geq
0$.  Then $\vx\geq\vec{0}$ if and only if the homogeneous component
$S_{\vx}$ is non-zero, and equivalently, if in the normal
form~\eqref{eq:normalform} of $\vx$ we have $m\geq 0$.

The image $\Oo$ of $S$ in $\mod^{\LL}(S)/\mod^{\LL}_0 (S)$ serves as
the structure sheaf of $\coh\XX$, and $\LL$ acts on the above data,
in particular on $\coh\XX$, by degree shift. Each line bundle has the
form $\Oo(\vx)$ for a uniquely determined $\vx$ in $\LL$, and we have
natural isomorphisms $$\Hom{}{\Oo(\vx)}{\Oo(\vy)}=S_{\vy-\vx}.$$
Defining the \emph{dualizing element} from $\LL$ as
$\vom=\vc-(\vx_1+\vx_2+\vx_3)$, the category $\coh\XX$ satisfies Serre
duality in the form $$\dual\Ext{1}{}{X}{Y}=\Hom{}{Y}{X(\vom)}$$
functorially in $X$ and $Y$. Moreover, Serre duality implies the
existence of almost split sequences for $\coh\XX$ with the
Auslander-Reiten translation $\tau$ given by the shift with $\vom$.

It is now established that an algebraic analysis of an $\LL$-graded singularity $S$ focusses on the \emph{singularity category} $\Sing{\LL}{S}={\Der{\modgr{\LL}{S}}}/{\Der{\projgr{\LL}{S}}}$. The singularity category is a good measure for the complexity of a singularity; in particular, $\Sing{\LL}{S}=0$ if and only if $S$ has finite (graded) global dimension, see~\cite{Buchweitz} for the ungraded and~\cite{Orlov:2009} for the $\ZZ$-graded case.

An often more accessible incarnation of $\Sing{\LL}{S}$ is given by the stable category $\sCMgr{\LL}{S}$ of $\LL$-graded Cohen-Macaulay modules (provided $S$ is graded Gorenstein, a hypothesis satisfied in our case.) In our case a more tractable version is available in the form of the stable category $\sMFgr{\LL}{S}$ of \emph{matrix factorizations}. A competing representation is provided by the fact that $\CMgr{\LL}{S}=\vect\XX$, yielding finally the most convenient access: the stable category $\svect\XX$ of vector bundles on $\XX$.
An introduction to CM-modules and matrix factorizations is found in \cite{Yoshino:1990}. As an introduction to singularity theory we recommend the book of Ebeling~\cite{Ebeling:book}. Concerning links between representation theory and singularities we refer to the introductory article~\cite{Reiten:1999} by I.~Reiten.

The category $\vect\XX$ carries the structure of a Frobenius
category such that the system $\Ll$ of all line bundles is the system
of all indecomposable projective-injectives,
see~\cite{Kussin:Lenzing:Meltzer:2010pre}: A sequence $\eta\colon 0\ra
E'\up{\alpha} E\up{\beta} E''\ra 0$ in $\vect\XX$ is \emph{distinguished exact} if
all the sequences $\Hom{}{L}{\eta}$ with $L$ a line bundle are exact
(equivalently all the sequences $\Hom{}{\eta}{L}$ are exact). In this case we say that $\alpha$ (resp.\ $\beta$) is a \emph{distinguished monomorphism} (resp.\ \emph{distinguished epimorphism}).
Each distinguished exact sequence is further exact in the abelian category $\coh\XX$.

By~\cite{Happel}, the stable category $$\svect\XX=\vect\XX/[\Ll]$$
therefore is triangulated. It is shown in~\cite{Kussin:Lenzing:Meltzer:2010pre}
that the triangulated category $\svect\XX$ is Krull-Schmidt with Serre
duality induced
from the Serre duality of $\coh\XX$. The triangulated category
$\svect\XX$ is
homologically finite. Moreover, we will see in Appendix~\ref{sect:appendix} that $\svect\XX$ has
a tilting object.

It is shown in~\cite{Geigle:Lenzing:1987} that the quotient functor
$q\colon\mod^{\LL}(S)\ra\coh\XX$ induces an equivalence
$\CM^{\LL}(S)\up{\sim}\vect\XX$, where $\CM^{\LL}(S)$ denotes the
category of $\LL$-graded (maximal) Cohen-Macaulay modules over $S$.
Under this equivalence indecomposable graded projective modules over
$S$ correspond to line bundles in $\vect\XX$, resulting in a natural
equivalence
$$\sCM^{\LL}(S)\simeq\svect\XX$$
used from now on as an identification.
Stable categories of (graded) Cohen-Macaulay modules play an important
role in the analysis of singularities,
see~\cite{Buchweitz,KST-1,KST-2,Orlov:2009}.

Each coherent sheaf on $\XX$ has the form $X=E\oplus X_0$ where $X_0$ is the largest subobject of finite length (then having finite support in $\XX$) and $E$ is a vector bundle. By $\vect\XX$ we denote the full subcategory of $\coh\XX$ given by all vector bundles.
Let $\delta:\LL\ra\ZZ$ be the homomorphism given on the generators $\vx_i$, $i=1,2,3$, by $\delta(\vx_i)=\bp/p_i$, where $\bp=\lcm(p_1,p_2,p_3)$. There are two important linear forms on the \emph{Grothendieck group} $\Knull{\coh\XX}$ of $\coh\XX$, called \emph{rank} and \emph{degree}.
The rank (degree) is characterized by the fact that $\rank{\Oo(\vx)}=1$ (resp.\ $\deg{\Oo(\vx)}=\delta(\vx)$) for all $\vx$ from $\LL$, see~\cite{Geigle:Lenzing:1987}. The rank is zero on all sheaves of finite length and $>0$ on non-zero vector bundles; further the degree is $>0$ on each simple sheaf. These basic properties show that any non-zero $X$ from $\coh\XX$ has non-zero rank or non-zero degree, implying that the \emph{slope} $\slope{X}=\deg{X}/\rank{X}$ is a properly defined member of $\QQ\cup\set{\infty}$. The slope $\slope{X}$ gives a rough information on the position of $X$ in the category $\coh\XX$ since $\Hom{}{X}{Y}$ is non-zero (resp.\ zero) if $\slope{Y}-\slope{X}$ (resp.\ $\slope{X}-\slope{Y}$) is large. Further we will need a refinement of the degree, the \emph{determinant} homomorphism $\determ\colon\Knull{\coh\XX}\ra \LL$ characterized by $\determ{\Oo(\vx)}=\vx$ for all $\vx\in\LL$, see~\cite{Lenzing:Meltzer:1993}. Note that $\deg{X}=\delta(\det{X})$.

\section{Projective covers and Auslander bundles} \label{sect:covers}

Since $S$ is $\LL$-graded
local with maximal graded ideal $(x_1,x_2,x_3)$, the category
$\modgr{\LL}{S}$, hence also $\CMgr{\LL}{S}=\vect{\XX}$ has projective
covers. Translating from Cohen-Macaulay modules to vector bundles,
we have to exhibit an \emph{irredundant} system of \emph{generators for the functor} $\Ll(-,E)$. That is, we have to find morphisms $L_i\stackrel{u_i}\ra
E$, $L_i\in\Ll$, $i=1,\ldots,n$ such that each morphism $f\colon L\ra E$, with
$L\in\Ll$, has the form $f=\sum_{i=1}^n u_i \alpha_i$ for some
$\alpha_i\colon L\ra L_i$ and, moreover, no proper subsystem of $u_1,\ldots,u_n$ has this property. In this case $u=(u_1,\dots,u_n)\colon\bigoplus
L_i\ra E$ is the projective cover of $E$ in $\vect{\XX}$. Existence for and properties of injective hulls follow from their projective counterpart by applying vector bundle duality ${\vect\XX}^{op}\ra\vect\XX$, $X\mapsto \mathcal{H}om(X,\Oo)$.

\begin{proposition} \label{prop:irredundant}
Assume $E$ is a vector bundle, and $u_i\colon L_i\ra E$, $i=1,\ldots,n$ are morphisms with line bundles $L_i$. We put $u=(u_1,\ldots,u_n):\bigoplus_{i=1}^nL_i\ra E$, and denote by $F$ the kernel of $U$ and by $v=(v_1,\ldots,v_n)^t$ the inclusion of $U$ into $\bigoplus_{i=1}^n L_i$. Then $(u_1,\ldots,u_n)$ is a generating system for $\Ll(-,E)$ if and only if the the sequence
\begin{equation}
\mu: 0\lra F\up{v} \bigoplus_{i=1}^n L_i \up{u} E \lra 0
\end{equation}
is distinguished exact. In this case, the following assertions are equivalent:

(i) $E$ has no line bundle summand, and $(u_1,\ldots,u_n)$ is irredundant.

(ii) $F$ has no line bundle summand, and $(v_1,\ldots,v_n)$ is irredundant.
\end{proposition}
\begin{proof}
The first claim follows immediately from the definitions. Next we show $(i)\Rightarrow (ii)$.  Assume  that $(v_1,\ldots,v_n)$ is not irredundant and, say, $(v_1,\ldots,v_{n-1})$ is already generating $\Ll(F,-)$. With $P'=\bigoplus_{i=1}^{n-1}L_i$ there results a commutative diagram
$$
\xymatrix@C22pt@R14pt{
        &       &            &0            &     0    &  \\
\mu': & 0\ar[r] &F\ar[r]^{v'} & P' \ar[r]^{u'}\ar[u] & E''\ar[r]\ar[u]& 0\\
\mu:  & 0\ar[r] &F\ar@{=}[u]\ar[r]^{(v',v_n)^t} & P'\oplus L_n \ar[r]^{(u',u_n)}\ar[u]_{\pi_n}& E  \ar[r]\ar[u]_{\pi_n''}& 0\\
      &         &                               & L_n\ar@{=}[r]\ar[u]_{\iota_n} &L_n\ar[u]_{\iota_n''}& \\
      &         &                               &0\ar[u]& 0\ar[u]& & & \\
}
$$
where the rows are distinguished exact, and where the central column splits. It follows that the right column is distinguished and then splits, since $L_n$ is relative injective in $\vect\XX$. Hence $L_n$ is a line bundle summand of $E$, contradicting assumption (i).

Assume next that $L$ is a line bundle such that $F=F'\oplus L$, and write $v=(v',w)$. As a composition of two distinguished monomorphisms then $w=(w_1,\ldots,w_n)$, where $w_i:L\ra L_i$, is also a distinguished monomorphism, hence splits since $L$ is injective in $\vect\XX$. Writing $P=\bigoplus_{i=1}^nL_i$ there exists $p:P\ra L$, where $p=(p_1,\ldots,p_n)$, with $1_L=\sum_{i=1}^n p_iw_i$. One of the summands, say $p_nw_n$ must be non-zero. Thus composition $L\up{w_n}L_n\up{p_n}L$ is an isomorphism implying that $pr_n\circ w=w_n: L \ra L_n$ is an isomorphism where $pr_n:\bigoplus_{i=1}^nL_i\ra L_n$ denotes the $n$th projection. To simplify notation, we identify $L$ with the direct summand $w(L)$ of $P$. We have shown that $pr_n:P\ra L_n$ restricts to an isomorphism $L_n\ra L$, yielding a splitting $P=L\oplus \bigoplus_{i=1}^{n-1}L_i$. We may then rewrite $\mu$ as
$$
\def\mat{{\arraycolsep2pt\renewcommand\arraystretch{0.4}\left(\begin{array}{cc}*&0\\ *&1_L\end{array}\right)}}
\xymatrix{0\ar[r]& F'\oplus L \ar[rr]^(.4){\mat}&& \left(\bigoplus_{i=1}^{n-1}L_i\right)\oplus L\ar[rr]^(.62){(u_1,\ldots,u_{n-1},u_n)}&&E\ar[r]& 0}
$$
where $u'=(u_1,\ldots,u_{n-1})$. Exactness now implies $u_n=0$, hence $(u_1,\ldots,u_{n-1})$   generates $\Ll(-,E)$ thus yielding a contradiction. This finishes the proof of $(i)\Rightarrow (ii)$. The implication $(ii)\rightarrow(i)$ follows from the previous one by vector bundle duality $\vect\XX \ra \vect\XX$, $X\mapsto\check{X}=\mathcal{H}om(X,\Oo)$, by observing that the duality preserves line bundles and exact sequences.
\end{proof}
A morphism $u\colon P\ra E$ with $P=\bigoplus_{i=1}^n L_i$, $L_i\in\Ll$, projective is called a \emph{projective cover} or hull of $E$ if $u$ is a distinguished epimorphism, and each morphism $h:X\ra E$ such that $u\circ h$ is a distinguished epimorphism is itself a distinguished epimorphism. The projective cover of $E$ is unique up to isomorphism. The concept of an \emph{injective hull} is dual.
\begin{corollary} \label{cor:cover_hull_duality}
In the absence of line bundle summands for $E$ or $F$ in \eqref{eq:proj-cover-ausl-bundle-L}, the injective hull of $E$ equals the projective hull of $F$. Moreover,
assume $E$ and $F$ have rank $\geq2$. Then $E$ is indecomposable if and only if $F$ is indecomposable.
\end{corollary}
\begin{proof}
We only need to show the last assertion. Assume for instance that $E$ is indecomposable, and $F$ has a non-trivial decomposition $F=F'\oplus F''$. Then, using $\Im$ for the injective hull, we have $E=\Im(F)/F=\Im(F')/F'\oplus\Im(F'')/F''$ implying that $E$ is decomposable, contrary to our assumption.
\end{proof}
For later use we note the following surprising result.
\begin{proposition} \label{prop:surprising}
Let $F$ be a vector bundle with injective hull $\xymatrix@!C=9pt{F\ar[rr]^(.35){(v_1,\ldots,v_n)}&&\bigoplus_{i=1}^n L_i.}$ Then each $v_i\colon F\ra L_i$ is an epimorphism in $\coh\XX$.
\end{proposition}
\begin{proof}
For $i=1,\ldots,n$ let $L_i'$ be the image of $v_i$ in $L_i$. Note that $L'_i$ is again a line bundle. Let $v_i'\colon F\ra L'_i$ be the morphism induced by $v_i$, and put $v'=(v'_1,\ldots,v'_n)$. Clearly, $F\up{v'}\bigoplus_{i=1}^n L'_i$ is a distinguished monomorphism, then yielding a distinguished monomorphism $h:\bigoplus_{i=1}^n L_i\ra \bigoplus_{i=1}^n L'_i$ by the defining property of the injective hull $\Im(F)$.  Passing to degrees we obtain
$$
\sum_{i=1}^n \deg{L_i} \leq \sum_{i=1}^n\deg{L'_i} \leq \sum_{i=1}^n\deg{L_i}.
$$
The first inequality uses that the cokernel of $h$ has rank zero, and then finite length. The second inequality uses the inclusions $L_i'\incl L_i$.
We hence obtain $\deg{L'_i}=\deg{L_i}$ for each $i=1,\ldots,n$ and then $L'_i=L_i$, proving our claim.
\end{proof}

Nearly all information one has on the category $\vect{\XX}$ is
obtained by using the existence of line bundle filtrations for vector bundles, and
then invoking the very explicit knowledge one has for morphism and
extension spaces between line bundles~\cite{Geigle:Lenzing:1987}. Recall in this context that $\Hom{}{\Oo(\vx)}{\Oo(\vy)}=S_{\vy-\vx}$ and $\Ext1{}{\Oo(\vx)}{\Oo(\vy)}=\dual{S_{\vx-\vy+\vom}}$. Passage to the stable category
$\svect\XX$ kills the carriers of all this information. A (partial) replacement is
found in the Auslander bundles. By definition, the \emph{Auslander bundle} $E=E(L)$
is obtained as the extension term of the almost split
sequence
\begin{equation}\label{eq:almost-split-auslander}\eta:\ 0\lra L(\vom)\up\alpha E(L)\up\beta L\lra 0
\end{equation}
 where $L$ denotes a line bundle.
Note for this that $\dual\Ext{1}{}{L}{L(\vom)}=\End{}{L}=k$, such that $E$ is
uniquely determined up to isomorphism. The next result points to the importance of
having exactly three weights. (For the proof of this statement, as in~\cite{Geigle:Lenzing:1987} we temporarily allow an arbitrary number of weights.)

\begin{proposition} \label{prop:auslander:exceptional}
Let $t=t_\XX$ denote the number of weights of $\XX$. Then the Auslander bundle $E=E(L)$ has trivial endomorphism ring $\End{}{E}=k$ if and only if $t\geq3$. Moreover, $E$ is exceptional if and only if $t=3$. In this case
, $E$ is determined by its class $[E]$ in the Grothendieck group $\Knull{\coh\XX}$.
\end{proposition}

\begin{proof}
From the exact Hom-Ext sequence $(L,\eta)$ we first obtain $\Hom{}{L}{E}=0=\Ext1{}{L}{E}$ since the connecting homomorphism is an isomorphism. Next, from the exact Hom-Ext sequence $(L(\vom),\eta)$ we deduce  $\Hom{}{L(\vom)}{L}=k^{4-t}$ and $\Ext1{}{L(\vom)}{L}=k^{t-3}$ with the convention that $k^n=0$ for $n<0$. The above uses the normal form expressions $-\vom=\sum_{i=1}^t\vx_i+(2-t)\vc$ and $2\vom=\sum_{i=1}^t(p_i-2)\vx_i+(t-4)\vc$. Finally, application of $(-,E)$ to $\eta$ yields exactness of $0\ra (L(\vom),E)\ra (E,E)\ra (L,E) \ra {}^1(L(\vom),E)\ra{}^1(E,E)\ra{}^1(L,E)\ra 0$, hence $\End{}{E}=k^{4-t}$ and $\Ext1{}{E}{E}=k^{t-3}$. For the last assertion we refer to \cite[Prop.~4.4.1]{Meltzer:2004}.
\end{proof}
Except in cases $(2,2,n)$, the assignment $L\mapsto E(L)$ yields a natural bijection between line
bundles and Auslander bundles. We will show later (Corollary~\ref{prop:suspension}) that for three weights, all Auslander bundles are also
exceptional in $\svect{\XX}$.

\begin{lemma}\label{lemma:exists-map-from-AB}
We assume a weight triple. Let $X$ be an indecomposable bundle of rank $\geq 2$. Then there
exists an Auslander bundle $E$ and a morphism $u\colon E\ra X$ such
that $u\not\in [\Ll]$.
\end{lemma}
\begin{proof}
  Choose a line bundle $L'$ of maximal degree (=slope) such that there
  is a morphism $0\neq h'\colon L'\ra X$. The almost split sequence
  $\eta\colon 0\ra L'\up{\alpha}E\up{\beta}L'(-\vom)\ra 0$ yields a
  morphism $h\colon E\ra X$ with $h\alpha=h' \neq 0$. We show
  $h\not\in [\Ll]$: Otherwise there would be a factorization
  $h=\sum_{i=1}^n b_i a_i$ with morphisms $E\lup{a_i}L_i \lup{b_i} X$
  and line bundles $L_i$. Then we have $0\neq h\alpha=\sum_{i=1}^n b_i a_i \alpha$,
  yielding an index $i$ with non-zero composition $
   b_i a_i\alpha$. In particular $\slope{L'}\leq\slope{L_i}$ and
  $\Hom{}{L_i}{X}\neq 0$. By the choice of $L'$ we get
  $\mu{L'}=\slope{L}$, thus $a_i \alpha$ is an isomorphism. This
  implies that $\eta$ splits, a contradiction.
\end{proof}

\begin{proposition} \label{prop:cover_ausl_bundle}
We assume a weight triple. Let $E=E(L)$ be the Auslander bundle given by the sequence $0\ra L(\vom)\up{\alpha} E \up{\beta} L \ra 0$. Then there is a distinguished exact sequence
\begin{equation} \label{eq:proj-cover-ausl-bundle-L}
\mu:\xymatrix{ 0 \ar[r]& F \ar[rr]^(.3){(v_0,v_1,v_2,v_3)^t} &&L(\vom)\oplus \bigoplus_{i=1}^{3} L(-\vx_i)\ar[rr]^(.65){(u_0,u_1,u_2,u_3)} &&E\ar[r]&0}
\end{equation}
in $\vect\XX$ which defines the projective cover of $E$, where $F$ is a bundle of rank two, and exceptional in $\vect\XX$. Moreover, for weight type $(2,a,b)$ we have $F = E(-\vx_1)$.
\end{proposition}
We note that already for weight type $(3,3,3)$ the bundle $F$ is \emph{not} an Auslander bundle.
\begin{proof}
\underline{Step 1:}
Let $E$ be the Auslander bundle given by the almost-split sequence
 $\eta:\ 0\lra L(\vom)\stackrel{\alpha}\lra E\stackrel{\beta}\lra
L\lra 0.$ By the almost-split property, the maps $x_i\colon
L(-\vx_i)\ra L$ lift to maps $u_i\colon L(-\vx_i)\ra E$ (in fact
unique since
$\Ext{1}{}{L(-\vx_i)}{L(\vom)}\simeq\dual\Hom{}{L}{L(-\vx_i)}=0$). We
claim that the maps $u_0=\alpha\colon \overline{L}(\vom)\ra E$, $u_i\colon
L(-\vx_i)\ra E$ ($i=1,2,3$) form an irredundant system of generators of
$\Ll(-,E)$. Write $\overline{L}=L(-\vy)$, $\vy\in\LL$. Let
$f\colon\overline{L}\ra E$ be a map. The map $\beta\circ f\colon
L(-\vy)\ra L$ cannot be an isomorphism since $\eta$ does not
split. Hence involving $\Hom{}{L(-\vy)}{L}=S_{\vy}\subseteq
(x_1,x_2,x_3)$ we see $\beta\circ f=\sum_{i=1}^3 x_i f_i$ for some
$f_1,f_2,f_3$. Now, consider $g=f-\sum_{i=1}^3 u_i f_i$; we obviously
get $\beta\circ g=0$, so that $g$ has the form $g=\alpha\circ f_0 =u_0 \circ
f_0$. We have shown $f=\sum_{i=0}^3 u_i f_i$ establishing that
$u_0,\dots,u_3$ generate $\Ll(-,E)$. That the
system is irredundant follows from the fact that the line bundle summands of the
central term of \eqref{eq:proj-cover-ausl-bundle-L} are pairwise Hom-orthogonal.

\underline{Step 2:} We put $L_0=L(\vom)$ and $L_i=L(-\vx_i)$ for $i=1,2,3$. By Proposition~\ref{prop:irredundant} the sequence $\mu$ is distinguished exact and
then also exact in $\coh{\XX}$, implying that $F$ has rank $2$. We are
going to show that $F\simeq E(-\vx_1)$. Note that $F$ is
indecomposable; otherwise $F$ would be the direct sum of two line
bundles, hence $0$ in $\svect{\XX}$, contradicting $F\simeq E[-1]\neq
0$ in $\svect{\XX}$. We claim, moreover, that $F$ is exceptional in
$\coh{\XX}$, that is, satisfies $$\text{a)}\; \End{}{F}=k\quad \text{and}\quad
  b)\; \Ext{1}{}{F}{F}=0.$$ From $F\simeq E[-1]\in\svect{\XX}$ we
infer that $\sEnd{}{F}=k$. To prove a) it thus suffices to show that
each morphism $h\colon F\ra F$ having a factorization
$h=\alpha\circ v$ with $\alpha\colon\bigoplus_{i=0}^3L_i\ra F$, is already zero.  Indeed, otherwise
$v\circ\alpha\colon\bigoplus L_i\ra\bigoplus L_i$ is non-zero. Since $L_0,\ldots,L_3$ are pairwise Hom-orthogonal,  we obtain
$\End{}{\bigoplus L_i}=\bigoplus\End{}{L_i}$, and deduce that there exists an
$i\in\{0,\dots,3\}$ such that $v_i\circ\alpha_i\colon L_i\ra L_i$ is
non-zero, hence an isomorphism. Therefore $F$ decomposes, yielding a
contradiction.

Next we assume weight type $(2,a,b)$ and show $[F]=[E(-\vx_1)]$ holds in $\Knull{\coh\XX}$. There is a unique simple sheaf $S_2$ in the second exceptional point such that there is an exact sequence (1)~$0\ra L(-\vx_2)\up{x_2}L\ra S_2\ra0$, see~\cite{Geigle:Lenzing:1987}. By the assumption on the weight type we have $\vom-\vx_1=-\vx_2-\vx_3$. Hence, applying the shift by $-\vx_3$ to (1) yields an exact sequence (2)~$0\ra L(\vom-\vx_1)\up{x_2}L(-\vx_3)\ra S_2\ra 0$. Passing to classes in the Grothendieck group, we then obtain $[L]-[L(-\vx_2)]= [L(-\vx_3)]-[L(\vom-\vx_1)]$, and then invoking exactness of \eqref{eq:proj-cover-ausl-bundle-L} we obtain
$[E(-\vx_1)=[L(\vom-\vx_1)] = \sum_{i=1}^3[L(\vx_i)] -[L] =[F]$.
Since $E$ is exceptional and $[F]=[E(-\vx_1]$, then
  \begin{displaymath}
    1 = \euler{F}{F}=\dim\End{}{F}-\dim\Ext{1}{}{F}{F} = 1-\dim\Ext{1}{}{F}{F},
  \end{displaymath}
and $\Ext{1}{}{F}{F}=0$ follows. Thus $F$ is exceptional. Since exceptional objects $X$ in
$\coh{\XX}$ are determined by their class $[X]$, see~\cite[Prop.~4.4.1]{Meltzer:2004} it follows that $F\simeq E(-\vx_1)$.
\end{proof}

To show exceptionality of Auslander bundles in the triangulated category $\svect\XX$ the following argument is useful. For distinction we use the notation $\sHom{}{X}{Y}$ to denote morphism spaces in $\svect\XX$.
  \begin{lemma} \label{lemma:determinant}
  Assume $E$ and $F$ are vector bundles of rank two, and $u\colon E\ra F$ is nonzero in $\sHom{}{E}{F}\neq0$. Then $\determ(E)\leq\determ(F)$.
  \end{lemma}

  \begin{proof}
  Observe first that $u$ is a monomorphism, since otherwise the image of $u$ would be a line bundle. We thus obtain an exact sequence $0\ra E\ra F\ra C\ra 0$ with $C$ of finite length. Passage to determinants proves the claim since $\det{S}>0$ for each simple sheaf $S$,~\cite{Lenzing:Meltzer:1993}.
  \end{proof}

By \cite{Happel} the suspension [-1] in the stable category $\svect\XX$ is induced by taking the projective hull. With the above notations we thus obtain $E[-1]=F$.
\begin{proposition} \label{prop:suspension}
Assume a weight triple. Then the following holds:

(i) Each Auslander bundle is exceptional in $\svect\XX$.

(ii) The suspension functor $[2]$ and the degree shift $\sigma_0$ by $\vc$ are isomorphic as functors on $\svect\XX$.

(iii) For each Auslander bundle $E$ we have $\det{E[n]}-\det{E}=n\vc$ for each integer $n$.

(iv) Let $\sAa$ be the full subcategory of $\svect\XX$ formed by the Auslander bundles. For weight type $(2,a,b)$ the suspension functor $[1]$ and the degree shift $\sigma_1$ by $\vx_1$ are isomorphic as functors on $\sAa$.
\end{proposition}

\begin{proof}
The proof of assertion $(ii)$ invokes the theory of matrix factorizations for a hypersurface singularity $f$. It uses that the stable category of $\LL$-graded matrix-factorizations is naturally equivalent to the category $\sCMgr\LL{S}$, and yields that the second suspension is isomorphic to the shift by the degree of the singularity $f$, so in our case yields the degree shift by $\vc$, compare~\cite[Theorem~2.14]{KST-1} for the $\ZZ$-graded case.
Assertion (iii) is obtained by applying the determinant to the sequence~\eqref{eq:proj-cover-ausl-bundle-L}. Finally, assertion (iv) follows by observing that the explicit construction of $F=E[-1]$ by means of the sequencef~\eqref{eq:proj-cover-ausl-bundle-L} is functorial in $E$.

Concerning (i) we know already that $\End{}{E}=k$ such that $\sEnd{}{E}=k$ follows.  By Serre duality we further have $\sHom{}{E}{E[n]}=\dual{\sHom{}{E[n-1]}{E(\vom)}}$, and we have to prove that this expression is zero for each non-zero integer $n$. Assume for contradiction that it is non-zero for some integer $n\neq0$. Applying the determinant and using (iii)  implies that for an integer $n$ as above, the inequalities (a) $n\vc\geq0$ and (b) $(n-1)\vc\leq 2\vom$ hold. Now, (a) is violated for $n<0$ and (b) is violated for $n>0$, thus proving the claim.
\end{proof}

\section{Proofs} \label{sect:proofs}
From now on the weight type is always the triple $(p_1,p_2,p_3)=(2,3,p)$ with
$p\geq 2$. In this Section we provide the proofs for Theorem~A,
Theorem~C and Lemma~B. Note that only the proof for Lemma~B is
straightforward. By contrast the proofs for Theorems~A and~C are far
from obvious. Additionally they behave quite sensitive with respect to
a (re)arrangement of the steps involved.

\subsection*{Proof of Lemma~B}
To prepare the proof of Lemma~B we observe that the
category $\ulPp$ has the shape of an infinite ladder:
$$
\def\ci{\circ}
\xymatrix@R22pt@C22pt{
 \cdots\ar[r]&\Oo(-\vx_3) \ar[d]_{x_2}\ar[r]^{x_3} &\Oo
 \ar[d]_{x_2}\ar[r]^{x_3} & \Oo(\vx_3)\ar[d]_{x_2}\ar[r]^{x_3}
 &\Oo(2\vx_3)\ar[d]_{x_2}\ar[r]^{x_3}&\cdots\\
\cdots\ar[r]&\Oo(\vx_2-\vx_3) \ar[r]^{x_3}       &\Oo(\vx_2)
\ar[r]^{x_3}       &\Oo(\vx_2+\vx_3)\ar[r]^{x_3}&
\Oo(\vx_2+2\vx_3)\ar[r]^{x_3}&\cdots\\}
$$
where the \emph{upper bar} (resp.\ \emph{lower bar}) is formed by all
line bundles $\Oo(n\vx_3)$, (resp.\ $\Oo(\vx_2+n\vx_3)$) for an
arbitrary integer $n$.

Commutativity of the diagram~\eqref{eqn:ladder} follows from the
commutativity of $S$. Applying
$\Hom{\XX}{\Oo(\vx)}{\Oo(\vy)}=S_{\vy-\vx}$ it follows that each
morphism in $\Pp$, viewed as a full subcategory of $\vect\XX$, is a
linear combination of powers of $x_2$ and $x_3$. Next we observe that
$\Homstab{}{\Oo(\vx)}{\Oo(\vx+\vc)}=0$ holds for each $\vx\in\PP$.
Indeed $\Hom{\XX}{\Oo(\vx)}{\Oo(\vx+\vc)}$ is generated by $x_2^3$ and
$x_1^2$, moreover each of the two morphisms factors through a fading
line bundle ($\Oo(\vx+2\vx_2)$ and $\Oo(\vx+\vx_1)$, respectively).
Finally, we have $x_2x_3^{p-1}\neq0$ (and hence $x_3^{p-1}\neq0$) in
$\vect\XX/[\Ff]$ since there are no morphisms from $\Oo(\vx)$ to
$\Oo(\vx+\vx_2+(p-1)\vx_3)$ factoring through a fading line bundle.
Indeed, every $\vy\in\LL$ with $\vec{0}\leq \vy \leq \vx_2+(p-1)\vx_3$
is of the form $\vy=a\vx_2+b\vx_3$ with $a=0,1$ and $b=0,\ldots,p-1$,
implying that $\vy$ belongs to $\PP$.~\qed

\begin{lemma} \label{lem:ses} Let $L$ be a line bundle. Then for each
  integer $n\geq 1$ the following sequence is exact in $\coh\XX$.
\begin{equation} \label{eqn:pushout} \eta_n:0\lra L\lup{(x_1,x_2^n)^t}
L(\vx_1)\oplus L(n\vx_2)\lup{(-x_2^n,x_1)}L(\vx_1+n\vx_2)\lra
0.\end{equation}
\end{lemma}

\begin{proof}
  The exact sequence is obtained from the following pushout diagram in
  $\coh\XX$
  \small\begin{equation}
    \label{eq:pushout-diagram}
    \xymatrix@C16pt@R18pt{ & 0 \ar @{->}[d] & 0 \ar
    @{->}[d] & & \\ 0\ar @{->}[r]& L\ar @{->}[r]^-{x_2^n} \ar
    @{->}[d]_-{x_1}& L(n\vx_2) \ar @{->}[r] \ar @{->}[d]^-{x_1} & S_2^{(n)}
    \ar @{->}[r] \ar @{=}[d] &0\\
    0\ar @{->}[r]& L (\vx_1) \ar @{->}[r]^-{x_2^n} \ar @{->}[d] &
    L(\vx_1+n\vx_2)\ar @{->}[r] \ar @{->}[d] & S_2^{(n)} \ar
  @{->}[r]&0 \\ & S_1
    \ar @{=}[r] \ar @{->}[d] & S_1 \ar @{->}[d] & & \\ & 0 & 0 & & }
  \end{equation}\normalsize
  with $S_1$ a simple sheaf concentrated in $x_1$ and $S_2^{(n)}$ a
  sheaf of length $n$
  concentrated in $x_2$.
\end{proof}
Denote by $\check{E}=\mathcal{H}om(E,\Oo)$ the dual vector bundle, and note that $\Oo(\vx)=\Oo(-\vx)$ for $\vx\in\LL$.
\begin{lemma}
The functor $d:(\vect\XX)^{op}\lra \vect\XX$, $E\mapsto \check{E}(\vx_2)$,  defines a self-duality preserving the partition $\Ll=\Pp\sqcup\Ff$. In particular, $d$ induces self-dualities of $\vect\XX/[\Ff]$ and $\ulPp$.~\qed
\end{lemma}

For the further discussion our next result is of central
importance. It expresses a fundamental property of the partition
$\Ll=\Pp\sqcup\Ff$.

\begin{proposition} \label{prop:2-generated}
Let $L$ be a persistent line bundle. Then the following holds:

\begin{enumerate}

\item[(1)] The functor $\Ff(L,-)=\res{\Hom{}{L}{-}}{\Ff}$ is generated
  by $x_1,x_2^2$ if $L$
  belongs to the upper bar of $\Pp$ and by $x_1,x_2$ if $L$ belongs to the
  lower bar of $\Pp$. With the notation from \eqref{eqn:pushout} put $\eta=\eta_2$ (resp.\ $\eta=\eta_1$) if $L$ belongs to the upper (resp.\ lower) bar. With the exception of $L$, all terms of $\eta$ then belong to $\add\Ff$, and for each $F\in\add{\Ff}$ the sequence $\Hom{}{\eta}{F}$ is exact.

\item[(2)] The functor $\Ff(-,L)$ is generated by $x_1,x_2$ if $L$
  belongs to the upper bar of $\Pp$ and is generated by $x_1,x_2^2$ if
  $L$ belongs to the
  lower bar of $\Pp$.
\end{enumerate}
\end{proposition}
\begin{proof}
 By the preceding lemma it suffices to show
  assertion (1). Applying a suitable shift with $\vx\in\ZZ\vx_3$ we
  can assume that $L=\Oo$ or $L=\Oo(\vx_2)$. In the first case each
  morphism $\Oo\ra \Oo(\vy)$ with $\vy\in\FF$ factors through
  $\Oo\lup{(x_1,x_2^2)^t}\Oo(\vx_1)\oplus\Oo(2\vx_2)$. In case
  $L=\Oo(\vx_2)$ each such morphism $\Oo(\vx_2)\ra\Oo(\vy)$ factors
  through
  $\Oo(\vx_2)\lup{(x_1,x_2)^t}\Oo(\vx_1+\vx_2)\oplus\Oo(2\vx_2)$. The
   preceding lemma now yields the short exact sequences $\eta_2$ and
  $\eta_1$, respectively, whose middle and end terms are clearly
  fading. The exactness of the sequences $\Hom{}{\eta_i}{F}$, $i=1,2$,
  then immediately follows.
\end{proof}

\begin{proposition} \label{prop:exact}
  Let $\eta\colon 0\ra X'\up{\alpha} X\up{\beta} X''\ra 0$ be a
  distinguished exact sequence in $\vect\XX$. Then the sequence
  $$
  \Phi(\eta)\colon 0\lra \Phi(X')\lup{\;\alpha_*\;} \Phi(X)\lup{\;\beta_*\;}
  \Phi(X'')\lra 0
$$
is an exact sequence in $\mmod\ulPp$.
\end{proposition}

\begin{proof}
  \emph{$\alpha_*$ is injective:} Let $L$ be a persistent line bundle
  and $f'\colon L\ra X'$ a morphism with $\alpha f'\in[\Ff](L,X)$.
  Using Proposition~\ref{prop:2-generated} we obtain a commutative
  diagram with exact rows
\begin{equation}
  \xy\xymatrixcolsep{2pc}\xymatrix@C18pt@R18pt{ \eta\colon 0\ar @{->}[r]& X'\ar
  @{->}[r]^-{\alpha} & X \ar @{->}[r]^-{\beta} & X''
    \ar @{->}[r]  &0\\
    \ \ \ \ \ 0\ar @{->}[r]& L \ar @{->}[r] \ar @{->}[u]_-{f'} &
    L_1 \oplus L_2 \ar @{->}[r] \ar @{->}[u]_-{f} & L_3 \ar
  @{->}[r] \ar @{->}[u]_-{f''}&0 }
  \endxy
\end{equation}
where $L_1$, $L_2$ and $L_3$ belong to $\Ff$. Since the sequence
$\eta$ is distinguished exact in $\vect\XX$, the morphism $f''$ lifts
via $\beta$, so equivalently $f'$ extends to $L_1\oplus L_2$. Hence
$f'\in[\Ff](L,X')$, as claimed.

\bigskip\noindent\emph{$\ker{\beta_*\subseteq \image{\alpha_*}}$:}
Assume $L\in \Pp$ and $f\colon L\ra X$ satisfies $\beta
f\in[\Ff](L,X'')$. This yields a commutative diagram
\begin{equation}
  \xy\xymatrixcolsep{2pc}\xymatrix@C18pt@R18pt{ 0\ar @{->}[r]& X'\ar
  @{->}[r]^-{\alpha} & X \ar @{->}[r]^-{\beta} & X''
    \ar @{->}[r]  &0\\
     &  &
    L  \ar @{->}[r]_-{a} \ar @{->}[u]^-{f} & L_1 \oplus L_2 \ar
  @{->}[u]_-{b} \ar @{-->}[ul]^-{\bar{b}}_-{\circlearrowright} & }
  \endxy
\end{equation}
with $L_1,L_2\in\Ff$. Now $b$ lifts via $\beta$ to a morphism
$\bar{b}\colon L_1\oplus L_2\ra X$. It follows $\beta(f-\bar{b}a)=0$
and hence there exists $f'\colon L\ra X'$ with $\alpha f'=f-\bar{b}a$
implying $\alpha_*(f')=f$ in $\ulPp(L,X)$.

\bigskip\noindent\emph{$\beta_*$ is surjective:} This is obvious since
$\eta$ is distinguished exact, and then already the mapping
$\Hom{}{L}{\beta}\colon\Hom\XX{L}{X}\ra\Hom\XX{L}{X''}$ is surjective.
\end{proof}

\begin{proposition} \label{prop:image-in-Sp}
  For each $E$ from $\vect\XX$ the right $\ulPp$-module
  $\Phi(E)=\ulPp(-,E)$ is finitely presented, indeed finite
  dimensional. Moreover, for each persistent line bundle $L$ from the
  upper bar the morphism $x_2^*\colon\ulPp(L(\vx_2),E)\ra \ulPp(L,E)$,
  induced by $x_2\colon L\ra L(\vx_2)$, is a monomorphism.
\end{proposition}

\begin{proof}
In $\vect\XX$ we choose a distinguished exact sequence $0\ra E' \ra P
\ra E \ra 0$ with $P$ from $\add\Ll$. By Proposition~\ref{prop:exact}
the induced mapping $\Phi(P) \ra \Phi(E)$ is surjective. Moreover,
$\Phi(P)$ is a finitely generated projective module over $\ulPp$,
hence finite dimensional. This implies that $\Phi(E)$ is finite dimensional
and finitely presented.

Next we show that all maps $x_2^*\colon\ulPp(L(\vx_2),E)\ra
\ulPp(L,E)$ induced by $L\lup{x_2}L(\vx_2)$, where $L$ is persistent
from the upper bar, are injective. Let $f\colon L(\vx_2)\ra E$ be a
morphism such that $fx_2 \in[\Ff]$. By
Proposition~\ref{prop:2-generated} we get a commutative diagram
$$
\xy\xymatrixcolsep{2pc}\xymatrix@C18pt@R18pt{ L\ar @{->}[r]^-{x_2} \ar
   @{->}[d]_-{(x_1,x_2^2)^t} & L(\vx_2) \ar @{->}[d]^-{f}\\
   L(\vx_1)\oplus L(2\vx_2) \ar @{->}[r]_-{(g,h)} &
   E, } \endxy
 $$
 and hence $fx_2 =gx_1+hx_2^2$, that is, $(f-hx_2)x_2=gx_1$. Using the
 pushout property of diagram~\eqref{eq:pushout-diagram} (with $n=1$)
 we obtain a
 morphism $\ell\colon L(\vx_1 +2\vx_2)\ra E$ such that $\ell
 x_1=f-hx_2$, and $f\in [\Ff]$ follows.
\end{proof}

Together with Proposition~\ref{prop:exact} we get

\begin{corollary}
  Viewing $\Phi$ as a functor from the Frobenius category
  $\vect\XX$ to the Frobenius category $\nilop{p}$ the functor is
  exact, that is, $\Phi$ sends distinguished exact sequences to
  distinguished exact sequences.~\qed
\end{corollary}

\subsection*{The kernel of $\Phi$}
Next, we are going to show that the kernel of $\Phi$ agrees with the
ideal $[\Ff]$ of morphisms factoring through finite direct sums of fading
line bundles.

\begin{lemma}\label{lemma:cyclic}
  Let $\XX=\XX(2,3,p)$ with $p\geq2$. Then the factor group
  $\LL/\ZZ\vx_3$ is cyclic of order $6$ and generated by the class of
  $\vom$. Moreover both in $\tau$- and $\taumin$-direction, the
  $\tau$-orbit of any line bundle in $\vect\XX$
  consists of persistent and fading bundles according to the $6$-periodic
  pattern ${+}\,{-}\,{+}\,{-}\,{-}\,{-}$, where $+$ and $-$ stand for
  persistent and fading, respectively.
\end{lemma}
\begin{proof}
By construction $\LL/\ZZ\vx_3$ is the abelian group on generators
$\px_1,\px_2$ with relations $2\px_1=3\px_2=0$, hence $\LL/\ZZ\vx_3$ is
cyclic of order $6$. Further we have the following congruences modulo
$\ZZ\vx_3$:
\begin{equation}\label{eqn:cyclic}
0\vom\equiv 0,\;\; 1\vom\equiv \vx_1+2\vx_2,\;\; 2\vom\equiv\vx_2,\;\;
3\vom\equiv\vx_1,\;\; 4\vom\equiv2\vx_2,\;\; 5\vom\equiv \vx_1+\vx_2
\end{equation}
which immediately implies the last claim.
\end{proof}
\begin{lemma}\label{lemma:persistent-direct-summand}
  Let $E$ be an Auslander bundle. Then there exists a persistent line
  bundle which is a direct summand of the projective cover $P(E)$ of
  $E$, equivalently we have $\Phi E\neq 0$.
\end{lemma}
\begin{proof}
  Let $L=\Oo$, then by Proposition~\ref{prop:cover_ausl_bundle} the projective cover of the Auslander bundle $E(\Oo)$ in
  $\vect\XX$  is given by the expression
  $P(E(\Oo))=\Oo(\vom)\oplus\bigoplus_{i=1}^3 \Oo(-\vx_i)$,
  and $\Oo(-\vx_3)$ is persistent. The assertion clearly also holds
  for $L=\Oo(n\vx_3)$ ($n\in\mathbb{Z}$). By the preceding lemma
  it then suffices to show that after
  twisting the expression for $P(E(\Oo))$ with $i\vom$ (for
  $i=1,\dots,5$) there will always exist a persistent line bundle on
  the right hand side.
  \begin{table}[h!]
    \centering
    $$\begin{array}{|c||c|c|c|c|}
        \hline
        \vx & \vx+\vom & \vx-\vx_1 & \vx-\vx_2 & \vx-\vx_3 \\
       \hline
       0\vom & \vx_1+2\vx_2 & \vx_1 & 2\vx_2 & \text{\fbox{$0$}}\\
       1\vom & \text{\fbox{$\vx_2$}} & 2\vx_2 & \vx_1 +\vx_2 & \vx_1 +2\vx_2 \\
       2\vom & \vx_1 & \vx_1+\vx_2 & \text{\fbox{$0$}} &
        \text{\fbox{$\vx_2$}}\\
       3\vom & 2\vx_2 & \text{\fbox{$0$}} & \vx_1 +2\vx_2 & \vx_1\\
       4\vom & \vx_1+\vx_2 & \vx_1+2\vx_2 & \text{\fbox{$\vx_2$}} & 2\vx_2\\
       5\vom & \text{\fbox{$0$}} & \text{\fbox{$\vx_2$}} & \vx_1 &
        \vx_1+\vx_2\\
        \hline
    \end{array}$$
    \caption{Persistent direct summands of $P(E)$}
    \label{tab:persistent-summands}
  \end{table}
  This follows from Table~\ref{tab:persistent-summands} with entries
  from $\LL$ modulo $\ZZ\vx_3$, where elements from $\PP$ are boxed.
 Since each row in the table contains an element from
  $\PP$ the claim follows.
\end{proof}

\begin{lemma}\label{lemma:zero-phi-of-map-from-AB}
  Let $E$ be an Auslander bundle and $u\colon E\ra X$ be a morphism in
  $\vect\XX$ with $\Phi u=0$. Then $u\in [\Ff]$.
\end{lemma}
\begin{proof} We divide the proof into several steps. Note that step~(1) and (2) hold for general weight triples, and only step~(3) requests weight type $(2,3,p)$.

  (1) Let $E=E(L)$ as in~\eqref{eq:almost-split-auslander}.
  By~\eqref{eq:proj-cover-ausl-bundle-L} the projective cover of $E$ is
  given by the expression $P(E)=L(\vom)\oplus\bigoplus_{i=1}^3 L(-\vx_i)$.
  By
  Lemma~\ref{lemma:persistent-direct-summand} at least one of the line
  bundles $L(\vom)$, $L(-\vx_1)$, $L(-\vx_2)$, $L(-\vx_3)$ is
  persistent.

  (2) \emph{Claim.} Let $\vy\in\{-\vom,\vx_1,\vx_2,\vx_3\}$. Then
  there is an exact sequence
  \begin{equation}
    \label{eq:ses2}
    0\lra
  L(-\vy)\lup{\;\alpha'\;}E\lup{\;\beta'\;}L(\vy+\vom)\lra
  0.
  \end{equation}
  Indeed, if $\vy=-\vom$, we take the almost split
  sequence~\eqref{eq:almost-split-auslander}. If $\vy=\vx_i$, then
  we are going to show that there is an exact sequence
  \begin{equation}
    \label{eq:ses}
    0\lra
  L(-\vx_i)\lup{\;\pi_i\;}E\lup{\;\kappa_i\;}L(\vx_i +\vom)\lra 0,
  \end{equation}
  where $\pi_i$ is induced by $x_i \colon L(-\vx_i)\ra L$ such that
  $\beta\pi_i =x_i$, and similarly $\kappa_i$ is such that $\kappa_i
  \alpha=x_i\colon L(\vom)\ra L(\vom+\vx_i)$. Indeed we have $\kappa_i
  \pi_i=0$, since $\Hom{}{L(-\vx_i)}{L(\vx_i +\vom)}=0$. By Proposition~\ref{prop:surprising} the morphism $\kappa_i$ is an epimorphism. As a
  non-zero map from a line bundle the map $\pi_i$
  is a monomorphism. We put $U=\kernel{\kappa_i}/\image{\pi_i}$. Since rank and degree are additive on the sequence~\eqref{eq:ses} we obtain $U=0$ and hence the exactness of \eqref{eq:ses}.

(3) Let $\vy\in\{-\vom,\vx_1,\vx_2,\vx_3\}$ be such that $L(-\vy)$ is
persistent, by step (1). It follows
that there is a short exact sequence $$0\lra L(-\vy)\up{a} L_1 \oplus
L_2 \up{b} L_3 \lra 0$$
with fading line bundles $L_1,L_2,L_3$, and satisfying the properties
of Proposition~\ref{prop:2-generated}.
Since $\Phi u=0$ and thus $u\alpha' \in [\Ff]$ we obtain a commutative square
$$
\xymatrix@C18pt@R18pt{ E \ar @{->}[r]^-{u} & X\\
  L(-\vy) \ar @{->}[u]^-{\alpha'} \ar @{->}[r]^-{a} & L_1 \oplus L_2
  \ar @{->}[u]_-{c}. }
$$
We next form the pushout diagram
$$
\xy\xymatrixcolsep{1.5pc}\xymatrix@C18pt@R18pt{ & 0 & 0 & & \\ & L(\vy+\vom) \ar
  @{=}[r] \ar @{->}[u] & L(\vy+\vom) \ar @{->}[u] & & \\
  0\ar @{->}[r] & E\ar @{->}[r] \ar @{->}[u]^-{\beta'} & \overline{E}
  \ar @{->}[r] \ar @{->}[u] & L_3
  \ar @{->}[r] &0\\
  0\ar @{->}[r]& L (-\vy) \ar @{->}[r]^-{a} \ar @{->}[u]^-{\alpha'} &
  L_1 \oplus L_2 \ar @{->}[r]^-{b} \ar @{->}[u] & L_3 \ar @{->}[r] \ar
  @{=}[u] &0 \\ & 0 \ar @{->}[u] & 0 \ar @{->}[u] & & } \endxy
$$
Since $u$ factors through $\overline{E}$, it is sufficient to show
that $\overline{E}\in\add{\Ff}$.

  One checks easily that $L(\vy+\vom)$ is fading. (For $\vy=-\vom$
  this follows from the $6$-periodic pattern in
  Lemma~\ref{lemma:cyclic}.) Therefore it is sufficient to show, that
  $\Ext{1}{}{L(\vy+\vom)}{L_i}=0$, by Serre duality equivalently,
  that $\Hom{}{L_i}{L(\vy+2\vom)}=0$ (for $i=1,2$).

  By Proposition~\ref{prop:2-generated} we can assume that $L_1
  =L(-\vy+\vx_1)$, and $L_2=L(-\vy+2\vx_2)$ if $L(-\vy)$ is from the
  upper bar, and $L_2=L(-\vy+\vx_2)$ if $L(-\vy)$ is from the lower
  bar. Therefore, one has to check whether
  $\Hom{}{\Oo}{\Oo(2\vy+2\vom-\vx)}$ is zero, that is, whether
  $2\vy+2\vom-\vx\not\geq 0$ for $\vx\in\{\vx_1,\vx_2,2\vx_2\}$. There
  are two cases:

  \emph{1.~case.} Assume that $P(E)$ admits a direct summand $L(-\vy)$
  which is a persistent line bundle from the upper bar. In this case
  $\vx\in\{\vx_1,2\vx_2\}$. Table~\ref{tab:persistent-summands} shows
  that we can assume $L=\Oo(i\vom)$ for $i=0,\,2,\,3,\,5$, and the
  value of $\vy$ can also extracted from that table. In all these cases
  it is easy to see that the condition $2\vy+2\vom-\vx\not\geq 0$ is satisfied.

  \emph{2.~case.} Assume that each persistent line bundle summand of
  $P(E)$ is from the lower bar. In this case $\vx\in\{\vx_1,\vx_2\}$.
  Table~\ref{tab:persistent-summands} shows that we can assume
  $L=\Oo(\vom)$ and $\vy=-\vom$, or $L=\Oo(4\vom)$ and $\vy=\vx_2$.
  In these cases again the condition $2\vy+2\vom-\vx\not\geq 0$ holds.
\end{proof}

\begin{proposition}
  Let $X\in\vect\XX$ be indecomposable such that $\Phi X=0$. Then
  $X\in\Ff$.
\end{proposition}
\begin{proof}
  If $X$ is a line bundle this is clear. Assume $\rank X\geq 2$. By
  Lemma~\ref{lemma:exists-map-from-AB} we obtain a morphism $u\colon
  E\ra X$ where
  $E$ is an Auslander bundle and $u\not\in [\Ll]$. By
  Lemma~\ref{lemma:zero-phi-of-map-from-AB} we get $0\neq\Phi
  u\colon\Phi E\ra\Phi X$, in particular $\Phi X\neq 0$.
\end{proof}
Let $\si_i$, $i=1,2,3$, denote the degree shift (line bundle twist) $E\mapsto
E(\vx_i)$. Then the isomorphism classes of line bundles decompose into
$6$ orbits under the action of the group $\langle\si_3\rangle$.
\begin{corollary}
  Let $X$ be an indecomposable bundle of $\rank X\geq 2$. Then the
  projective cover
  $P(X)$ of $X$ admits a persistent line bundle as a direct
  summand. Moreover $P(X)$ admits line bundle summands from at least
  four pairwise distinct $\langle\si_3\rangle$-orbits.
\end{corollary}
\begin{proof}
Since $\Phi(X)\neq0$ if and only if $P(X)$ contains a persistent line
bundle, the first claim immediately follows.
Concerning the last claim we recall that the class of $\vom$ generates $\LL/\ZZ\vx_3$ which
is cyclic of order $6$ and that the classes of $0$ and $2\vom$
represent the persistent members of $\LL$. For each integer $n$,
then also  $P(X)(n\vom)=P(X(n\vom))$ contains a persistent line
bundle. Let $U$ be the subset of
$\LL/\ZZ\vx_3$ corresponding to the $\langle\si_3\rangle$-orbits of
line bundles in $P(X)$. By the above, for each integer $n$ the set $U$
must contain $n$ or $n+2$. As is easily checked this implies $|U|\geq
4$, proving the claim.
\end{proof}
Actually there are, up to cyclic permutation, just two possibilities
for a four-element subset $U$ as above, given by the two following
patterns, where a black dot indicates membership in $U$.
$$
\def\b{\bullet}
\def\c{\circ}
\xymatrix@-1pc@!R=7pt@!C=4pt{
                         &\b\ar@{-}[r] & \b\ar@{-}[rd] &      \\
\c\ar@{-}[ru]\ar@{-}[rd] &             &               &  \c  \\
                         &\b\ar@{-}[r] & \b\ar@{-}[ru] &      \\
} \quad \quad
\xymatrix@-1pc@!R=7pt@!C=4pt{
                         &\b\ar@{-}[r] & \b\ar@{-}[rd] &      \\
\b\ar@{-}[ru]\ar@{-}[rd] &             &               &  \b  \\
                         &\c\ar@{-}[r] & \c\ar@{-}[ru] &      \\
}
$$
\begin{lemma}
  Assume $P\in\add\Ll$. There exists an exact sequence $\eta\colon
  0\ra P \up{\alpha} P_0 \up{\beta} P_1 \ra 0$ in $\coh\XX$ with
  $P_0,\,P_1$ from $\add\Ff$ such that:
\begin{enumerate}
\item for each persistent line bundle $L'$ the
sequence $\Hom{}{L'}{\eta}$ is exact;
\item For each fading line bundle $L'$ the sequence $\Hom{}{\eta}{L'}$
  is exact.
\end{enumerate}
\end{lemma}
\begin{proof}
  It suffices to show the statement if $P=L$ is indecomposable.  If
  $L\in\Ff$, then one can take $\eta\colon 0\ra L\ra L \ra 0\ra 0$.
  Let now $L\in\Pp$. Then let $$\eta=\eta_n\colon 0\lra
  L\lup{(x_1,x_2^n)^t} L(\vx_1)\oplus
  L(n\vx_2)\lup{(-x_2^n,x_1)}L(\vx_1+n\vx_2)\lra 0$$
  be one
  of the sequences from Proposition~\ref{prop:2-generated}, where
  $n=2$ if $L$ is from the upper bar and $n=1$ if $L$ is from the
  lower bar. Condition~(2) follows from~\ref{prop:2-generated}. Let $L'$ be a persistent line bundle and $h\colon L'\ra L(\vx_1+n\vx_2)$ a morphism. Without loss of generality assume that $h\neq 0$. By considering the four possible cases $n=1$ or $n=2$ and $L'$ from the upper or from the lower bar, respectively, one shows that $h\in\Hom{}{L'}{L(\vx_1+n\vx_2)}=x_1 \Hom{}{L'}{L(n\vx_2)}$, in particular $h$ factors through the middle term of $\eta$. This shows condition~(1).
\end{proof}
The next result constitutes a key step in our proof of Theorems~A and~C.
\begin{proposition} \label{prop:kernel}
  Each morphism $h\colon E\ra F$ in $\vect\XX$ with $\Phi(h)=0$ belongs
  to the ideal $[\Ff]$, that is, $h$ factors through a member of $\add\Ff$.
\end{proposition}

\begin{proof}
  Let $P\up{\pi}E\ra 0$
  be a distinguished epimorphism with
  $P\in\add{\Ll}$. Since $\Phi h=0$, the composition $h\pi$ factors
  through an object of $\add{\Ff}$, and by (2) from the preceding lemma we
  obtain a commutative diagram
  $$\xy\xymatrixcolsep{2pc}\xymatrix@C18pt@R18pt{\eta\colon 0\ar @{->}[r] & P \ar
    @{->}[r]^-{\alpha} \ar @{->}[d]_-{h\pi} & P_0 \ar
    @{->}[r]^-{\beta} \ar @{-->}[dl]^{\ga}_-{\circlearrowright} &
    P_1 \ar @{->}[r] & 0\\
    & F & & } \endxy$$
  where $\eta$ is an exact sequence in $\coh\XX$
  with $P_0$, $P_1\in\add{\Ff}$. From this we get a commutative diagram
    $$\xy\xymatrixcolsep{2pc}\xymatrix@C18pt@R18pt{\mu\colon 0\ar @{->}[r] & P \ar
    @{->}[r]^-{(\pi,\alpha)^t} & E\oplus P_0 \ar
    @{->}[r]^-{(\si_1,\si_2)} \ar @{->}[d]_-{(h,-\ga)} & C
    \ar @{->}[r] \ar @{-->}[dl]^-{\delta}_-{\circlearrowright} & 0\\
    & & F & } \endxy$$
  in $\coh\XX$ whose row is exact. It suffices to show that the
  cokernel $C$ lies in $\add{\Ff}$. To prove this consider the following
  commutative diagram
  $$
  \xy\xymatrixcolsep{2pc}\xymatrix{ & 0 & 0 & 0 & \\
    \eta\colon 0\ar @{->}[r]& P \ar @{->}[r]^-{\alpha} \ar @{->}[u] &
    P_0 \ar @{->}[r]^-{\beta} \ar
    @{->}[u] & P_1 \ar @{->}[r] \ar @{->}[u] &0\\
    \mu\colon 0\ar @{->}[r]& P \ar
    @{->}[r]^-{\varepsilon=(\pi,\alpha)^t} \ar @{=}[u] & E\oplus P_0
    \ar @{->}[r]^-{\si=(\si_1,\si_2)} \ar @{->}[u]_-{(0,1)} &
    C \ar @{->}[r] \ar
    @{-->}[u]_-{p} &0\\
    0 \ar @{->}[r] & 0 \ar @{->}[u] \ar @{->}[r] & E \ar @{=}[r] \ar
    @{->}[u]_-{(1,0)^t} &
    E \ar @{->}[r] \ar @{->}[u] & 0 \\
    & 0 \ar @{->}[u] & 0 \ar @{->}[u] & 0 \ar @{->}[u] & } \endxy
  $$
  with exact rows and columns.
  For each persistent line bundle $L$, applying the functor
  $\Hom{}{L}{-}$ the first (compare part (1) of the preceding lemma)
  and the third row, and the first and the second column stay exact.
  It follows that also the third column stays exact implying that
  $\Hom{}{L}{\mu}$ is exact for each $L\in\Pp$, in particular
  $\Hom{}{L}{\si}$ is an epimorphism for $L\in\Pp$. We conclude that
  $\Phi\si\colon\ulPp(-,E\oplus P_0)\ra\ulPp(-,C)$ is
  an epimorphism. Since $\Phi\varepsilon=\Phi\pi$ is also an
  epimorphism, the composition
  $0=\Phi(\si\varepsilon)=\Phi(\si)\Phi(\varepsilon)$ is an
  epimorphism as well, yielding $\Phi C=0$, equivalently $C\in\add{\Ff}$.
\end{proof}

\subsection*{The functor $\Phi$ is full}
Our next lemma plays a key role in order to show that the functor
$\Phi\colon\vect\XX\ra \mmod{\ulPp}$ is full.

\begin{lemma}
  We assume that $\eta\colon 0\ra E' \up{\alpha} E \up{\beta} E'' \ra 0$
  is a distinguished exact sequence in $\vect\XX$. Then
  $\beta=\coker(\alpha)$ holds in $\vect\XX/[\Ff]$.
\end{lemma}
\begin{proof}
  (1) $\beta$ is an epimorphism in $\vect\XX/[\Ff]$: To this end let
  $f\colon E''\ra X$ be a morphism in $\vect\XX$ such that
  $f\beta\in\ker\Phi$. Then $(\Phi f)(\Phi\beta)=0$. By
  Proposition~\ref{prop:exact}, $\Phi\beta$ is an epimorphism, thus
  $\Phi f=0$, that is, $f\in\ker\Phi$, hence $f\in[\Ff]$ by Proposition~\ref{prop:kernel}.

  (2) Let $h\colon E\ra X$ be a morphism in $\vect\XX$ such that
  $h\alpha\in[\Ff]$.
  Hence there is $P\in\add{\Ff}$ such that $h\alpha=[E' \up{a} P\up{b}
  X]$. Since $\eta$ is distinguished exact there is a morphism
  $a'\colon E\ra P$ such that $a'\alpha =a$. We obtain $(h-ba')\alpha
  =0$. Since $\eta$ is exact, there is a morphism $h'\colon E''\ra X$
  with $h-ba'=h'\beta$, which leads to $h=h'\beta$ modulo
  $[\Ff]$.
\end{proof}

\begin{proposition} \label{prop:full}
  The functor $\Phi\colon\vect\XX\ra \mmod{\ulPp}$ is full, and
  induces a full embedding $\vect\XX/[\Ff]\incl \nilop{p}$.
\end{proposition}

\begin{proof}
  Let $h\colon\Phi E\ra\Phi F$ be a morphism in
  $\mmod{\ulPp}$. Consider projective covers in $\vect\XX$:
  \begin{gather*}
    0\lra E' \up{\alpha}P\up{\beta}E\lra 0,\\
    0\lra F' \up{\ga}Q\up{\delta}F\lra 0.\\
  \end{gather*}
  By Proposition~\ref{prop:exact} we get a commutative diagram with
  exact rows:
  $$
  \xy\xymatrixcolsep{2pc}\xymatrix@C20pt@R20pt{ 0\ar @{->}[r] & \Phi E' \ar
    @{->}[r]^-{\Phi\alpha} \ar @{->}[d]_-{h'} & \Phi P\ar
    @{->}[r]^-{\Phi\beta} \ar @{->}[d]_-{\bar{h}=\Phi u} & \Phi E\ar
    @{->}[r] \ar @{->}[d]^-{h} & 0\\ 0\ar @{->}[r] & \Phi F' \ar
    @{->}[r]^-{\Phi\ga} & \Phi Q\ar @{->}[r]^-{\Phi\delta} &\Phi F\ar
    @{->}[r] & 0. } \endxy
   $$
 We have $\Phi(\delta u\alpha)=\Phi(\delta)\Phi(u)\Phi(\alpha)=0$,
 hence $(\delta u)\alpha$ belongs to $[\Ff]$. By the preceding
 lemma there is a morphism $v\colon E\ra F$ with $v\beta=\delta
 u$ in $\vect\XX/[\Ff]$. Applying $\Phi$ we get $(h-\Phi v)\Phi\beta=0$. Since $\Phi\beta$
 is an epimorphism we get $h=\Phi(v)$.
\end{proof}

\subsection*{Reflecting exactness}
The next proposition turns out to be crucial in comparing the
exact structures of $\vect\XX$, $\vect\XX/[\Ff]$, $\mod{\ulPp}$ and
$\snilop{p}$.

\begin{proposition} \label{prop:reflect-exact}
  Let $\eta\colon 0 \ra E' \up{\alpha} E \up{\beta} E'' \ra 0$ be a
  sequence in $\vect\XX$ such that $\Phi(\eta)$ is exact in
  $\mmod{\ulPp}$. Modifying terms by adding suitable summands from
  $\add\Ff$, we can change $\eta$ to a distinguished exact sequence
  $\hat\eta$ in $\vect\XX$ such that $\Phi(\eta)$ and $\Phi(\hat\eta)$
  are isomorphic in $\mmod{\ulPp}$ and, accordingly, $\eta$ and
  $\hat\eta$ are isomorphic in $\vect\XX/[\Ff]$.
\end{proposition}

\begin{proof}
  Let $P\up{\pi}E'' \ra 0$ in $\vect\XX$ be a projective cover. Since
  $\Phi\beta$ is an epimorphism in $\vect\XX/[\Ff]$ the
  morphism $\pi$ can be lifted to $E$. That is, there is a morphism
  $\bar{\pi}\colon P\ra E$ such that $\beta\bar{\pi}-\pi$ factors through a bundle
  $P_f$ which is a direct sum of fading line bundles, say
  $\beta\bar{\pi}-\pi=[P\up{\bar{h}}P_f \up{h}E'']$. It then follows
  that $E\oplus P_f\lup{(\beta,-h)}E''$ is a distinguished
  epimorphism. If $K$ denotes its kernel we obtain a commutative
  diagram of distinguished exact sequences in $\vect\XX$:
  $$
  \xy\xymatrixcolsep{2pc}\xymatrix@C18pt@R18pt{\hat{\eta}\colon 0\ar @{->}[r] &
    K\ar @{->}[r]^-{\hat{\alpha}} & E\oplus P_f
    \ar @{->}[r]^-{(\beta,-h)} & E'' \ar @{->}[r] & 0\\
    \eta\colon 0 \ar @{->}[r] & E' \ar @{->}[r]^-{\alpha} \ar
    @{-->}[u]_-{\ga} & E \ar @{->}[r]^-{\beta} \ar
    @{->}[u]_-{(1,0)^t} & E'' \ar @{->}[r] \ar @{=}[u] & 0. }\endxy$$
  Applying $\Phi$ to this diagram we obtain that $\Phi((1,0)^t)$ and
  hence $\Phi(\ga)$ are isomorphisms. Since
  $\Phi\colon\vect\XX/[\Ff]\ra\mod{\ulPp}$ is a full embedding, $\ga$
  becomes an isomorphism in the factor category $\vect\XX/[\Ff]$.
\end{proof}

\subsection*{The functor $\Phi$ is dense}

The next lemma will serve as an induction step to prove that the
functor $\Phi\colon\vect\XX\ra \nilop{p}$ is dense.

\begin{lemma} \label{lemma:simple}
  Let $L$ be a persistent line bundle and $\eta\colon 0\ra
  L(\vom)\up{\alpha} E\up{\beta} L \ra 0$ the corresponding
  almost split sequence in $\vect\XX$. Application of $\Phi$ yields an
  exact sequence
  \begin{equation}
    \label{eq:ses-simple-cokernel}
    0\lra \Phi(E) \lup{\Phi(\beta)}\Phi(L) \up{\pi} S \lra 0
  \end{equation}
  in $\mmod{\ulPp}$, where $S$ is a simple module (not necessarily
  lying in $\nilop{p}$).
\end{lemma}

\begin{proof}
  By assumption $\XX$ has exactly three weights, therefore the
  Auslander bundle $E$ is indecomposable and hence
  $\Phi(\beta)\colon\Phi(E)\ra\Phi(L)$ is not an isomorphism since
  $\Phi$ induces a full embedding $\vect\XX/[\Ff]\incl \mod\ulPp$. The
  modules $\Phi(E)$ and $\Phi(L)$ have local endomorphism rings;
  moreover, $\Phi(L)$ is indecomposable projective. Denote by
  $\pi\colon\Phi(L)\ra S$ the natural projection on the simple top.
  Since the mapping $\Phi(\beta)$ belongs to the radical of
  $\mmod\ulPp$ we obtain $\pi\beta=0$.

  We claim that the map $\Phi(\beta)\colon\Phi(E)\ra \Phi(L)$ is
  injective. Indeed let $L_1$ be a persistent line bundle and $f\colon
  L_1\ra E$ such that $\Phi(\beta f)=0$. This yields a factorization
  $\beta f=[L_1\up{a}P\up{b}L]$ with $P$ from $\add\Ff$.  As a radical
  morphism $b$ then lifts via $\beta$, thus $b=\beta \bar{b}$ for some
  morphism $\bar{b}\colon P\ra E$. We obtain $\beta(f-\bar{b}a)=0$
  such that $f-\bar{b}a$ factors (via $\alpha$) over $L(\omega)$.
  Since $L\in\Pp$, equation~(\ref{eqn:cyclic}) from
  Lemma~\ref{lemma:cyclic} shows that $L(\vom)$ is fading. It follows
  that $f$ belongs to $[\Ff]$, proving the claim.

  By the preceding argument we obtain an exact sequence $0\ra
  \Phi(E)\ra \Phi(L)\ra C \ra 0$. We claim that the cokernel term $C$
  is a simple $\ulPp$-module. We first show that $C$ --- viewed as a
  representation of $\ulPp$ --- has support $\set{L}$, and hence is
  semisimple. For each persistent line bundle $L_1$, not isomorphic to
  $L$, each morphism $\ga\colon\Phi(L_1)\ra C$ lifts by projectivity
  of $\Phi(L_1)$ to a morphism $\Phi(u)\colon\Phi(L_1)\ra\Phi(L)$.
  Since $\eta$ is almost split the non-isomorphism $u\colon L_1\ra L$
  lifts via $\beta$, then implying that $\ga=0$. We have shown that
  $C\iso S^n$ where $S=S_L$ denotes the simple module concentrated in
  $L$. Moreover, $n\geq1$ since $C\neq0$. As an indecomposable
  projective module $\Phi(L)$ is local, and we conclude that $n=1$,
  implying that $C$ is simple.
\end{proof}

\begin{proposition} \label{prop:dense}
For each module $M$ in $\nilop{p}$ there exists a bundle $X$ such that
$\Phi(X)$ is isomorphic to $M$.
\end{proposition}

\begin{proof}
We argue by induction on the (finite) dimension $n$ of $M$. If $n=0$,
the assertion is evident. So assume that $n>0$. Then we obtain an
exact sequence $0\ra M' \ra M \ra S \ra 0$ in $\mmod\ulPp$, where $S$
is simple and $M'$ belongs to $\nilop{p}$. Invoking
Lemma~\ref{lemma:simple} we obtain an Auslander bundle $E=E(L)$ and a commutative diagram in
$\mod\ulPp$ with exact rows and columns
$$
\xy\xymatrixcolsep{1.5pc}\xymatrix@C18pt@R18pt{ & & 0\ar @{->}[d] & 0\ar @{->}[d] & \\
  & & \Phi E \ar @{=}[r] \ar @{->}[d] & \Phi E \ar
  @{->}[d]^-{\Phi\beta} & \\ \mu\colon 0\ar @{->}[r]& M' \ar
  @{->}[r] \ar @{=}[d]& \overline{M} \ar @{->}[r] \ar @{->}[d] &
  \Phi L
  \ar @{->}[r] \ar @{->}[d] &0\\
  \ \ \ \ \ 0\ar @{->}[r]& M' \ar @{->}[r] & M \ar @{->}[r] \ar
  @{->}[d] & S \ar
  @{->}[r] \ar @{->}[d] &0 \\
  & & 0 & 0 & } \endxy
$$
Since $\Phi L$ is projective the sequence $\mu$ splits yielding
$\overline{M}=M'\oplus\Phi(L)$. By induction $M'$ belongs to the image of
$\Phi$, say $M'=\Phi(F')$. Summarizing we
obtain an exact sequence
\begin{equation}
  \label{eq:ses-zu-M}
0\ra \Phi E\lup{(\Phi\beta,\Phi u')^t} \Phi L\oplus\Phi F' \lra M \ra 0
\end{equation}
in $\mmod\ulPp$. We put $\bx_i=\vx_i+\vom$ and form in $\vect\XX$ the injective hull
$$0\ra E\up{a}L\oplus\bigoplus_{i=1}^3 L(\bx_i)\lra E(\vx_1)\ra 0$$
of
$E$ (compare~\eqref{eq:proj-cover-ausl-bundle-L}), where
$a=(\beta,\kappa_1,\kappa_2,\kappa_3)^t$, with $\beta$ from the almost
split sequence~\eqref{eq:almost-split-auslander} and the $\kappa_i$
like in~\eqref{eq:ses}. We obtain in $\coh\XX$
the exact sequence $$\ga\colon 0\ra
E\lup{(\beta,(\bx_i),u')^t}L\oplus\bigoplus_{i=1}^3 L(\bx_i)\oplus F'\lra
C\ra 0,$$
which is distinguished exact in $\vect\XX$.  To prove this we simplify notation, and write $\ga$ as the exact sequence $0\ra E \up{u} F \up{v}C \ra 0$. By construction of $\ga$ each morphism of $E$ into a line bundle $L$ extends to $F$, hence the sequence $\Hom{}{\ga}{L}$ is exact. We show that $C$ is a vector bundle. Let $C_0$ denote the torsion part of $C$, then the natural morphism $n:C\ra C/C_0$ induces an isomorphism $\Hom{}{C/C_0}{L}\ra \Hom{}{C}{L}$ for each line bundle $L$. This implies that the sequence $\bar\ga: 0\ra E\up{u}F \up{n\circ v}C/C_0\ra 0$ also has the property that the sequence $\Hom{}{\bar\ga}{L}$ is exact for each line bundle $L$. Additionally $\bar\ga$ consists of vector bundles, which implies that  $\bar\ga$ is distinguished exact, in particular exact in $\coh\XX$. Comparison of $\ga$ and $\bar\ga$ now shows that $C$ and $C/C_0$ are isomorphic, hence $C$ is a vector bundle and $\ga$ is distinguished exact as claimed. There are
two cases:

\emph{1.\ case.} $L$ belongs to the upper bar. Then all line bundles
$L(\bx_i)$ are fading ($i=1,2,3$), and $$\Phi\ga\colon 0\ra \Phi
E\lup{(\Phi\beta,\Phi u')^t}\Phi L\oplus\Phi
F'\ra\Phi\bigl(C\bigr)\ra 0$$
is exact. Comparing this
with~\eqref{eq:ses-zu-M} we obtain $\Phi\bigl(C\bigr)\simeq M$.

\emph{2.\ case.} $L$ belongs to the lower bar. Then $L(\bx_1)$ is
persistent, and $L(\bx_2)$, $L(\bx_3)$ are fading. We obtain the
diagram
\small$$
\xymatrix@C40pt@R24pt{
& & 0 & 0 & \\
  0 \ar @{->}[r] & \Phi E \ar @{=}[d] \ar @{->}[r]^-{(\Phi\beta,\Phi u')^t} &
  \Phi(L)\oplus\Phi(F') \ar @{->}[r] \ar @{->}[u] & M \ar @{->}[u] \ar
  @{->}[r]  & 0 \\ 0\ar
  @{->}[r]& \Phi E \ar @{->}[r]^-{(\Phi\beta,\Phi\bx_1,\Phi u')^t} &
  \Phi(L)\oplus\Phi(L(\bx_1))\oplus\Phi(F') \ar @{->}[u]_-{proj.}
  \ar @{->}[r] & \Phi(C)
  \ar @{->}[r] \ar @{->}[u] & 0 \\
   &  & \Phi(L(\bx_1)) \ar @{->}[u] \ar @{=}[r] &
   \Phi(L(\bx_1)) \ar @{->}[u] &  \\
  & & 0 \ar @{->}[u] & 0 \ar @{->}[u] & }
$$\normalsize
The sequence $$0\lra \Phi L(\bx_1)\lra\Phi C\lra M\lra 0$$
is exact with all
terms lying in $\nilop{p}$. This sequence splits since $\Phi L(\bx_1)$ is
injective in $\nilop{p}$. We get $\Phi C=M\oplus \Phi L(\bx_1)$. Write
$C=\bigoplus_{i=1}^n C_i$ with all $C_i \in\vect\XX$ indecomposable.
Since $\Phi$ is full the $\Phi C_i \neq 0$ have local endomorphism
rings. Because the category $\snilop{p}$ is Krull-Schmidt, it follows that $M$ is the direct sum of some of the $\Phi
C_i$, hence $M$ lies in the image of $\Phi$.
\end{proof}

For later applications we need a related result:
\begin{proposition} \label{prop:simple:Auslander}
  Let $L$ be a persistent line bundle from the upper bar and let $S_L$
  be the simple right $\ulPp$-module concentrated in $L$. Then $S_L$
  belongs to $\nilop{p}$ and has the form
  $\Phi\bigl(E(L)(\vx_1)\bigr)$, where $E(L)$ denotes the Auslander
  bundle attached to $L$.

Moreover, each simple $\ulPp$-module belonging to $\nilop{p}$ has the
above form.
\end{proposition}

\begin{proof}
  This follows from the proof (1.~case) of
  Proposition~\ref{prop:dense} (with $M=S$ and hence $F'=0$).
\end{proof}

\subsection*{Frobenius structure and proof of Theorems~A and~C}

Define a sequence $0\ra E'\up{\alpha}E\up{\beta}E''\ra 0$ in
$\vect\XX/[\Ff]$ to be distinguished exact if it is isomorphic to a
sequence which is induced by a distinguished exact sequence in
$\vect\XX$.

We will prove now Theorems~A and~C. Part~(1) from Theorem~A was
already shown before, part~(3) is trivial.

By Propositions~\ref{prop:image-in-Sp}, \ref{prop:kernel},
\ref{prop:full} and~\ref{prop:dense} the assignment $E\mapsto
\ulPp(-,E)$ induces an equivalence of categories
$\Phi\colon\ffrac{\vect\XX}{[\Ff]}\lra \nilop{p}$. It follows from
Propositions~\ref{prop:exact} and~\ref{prop:reflect-exact} that a
sequence $\eta\colon 0\ra E'\up{\alpha}E\up{\beta}E''\ra 0$ in
$\vect\XX/[\Ff]$ is distinguished exact if and only if $\Phi(\eta)$ is
exact in $\nilop{p}$. It follows (via $\Phi$) that the distinguished
exact sequences give $\vect\XX/[\Ff]$ the structure of a Frobenius
category, and moreover, such that the indecomposable
projective-injective objects are given by the objects of $\ulPp$. This
proves part~(2) from Theorem~A. Hence
$\Phi\colon\ffrac{\vect\XX}{[\Ff]}\lra \nilop{p}$ is even an
equivalence of Frobenius categories, which shows the first statement
of Theorem~C. (We note that it is possible to establish directly that the distinguished exact
sequence define on $\svect\XX/[\Ff]$ the structure of a Frobenius category, without
involving the functor $\Phi$.) The second statement of Theorem~C is an immediate
consequence of the first together with Theorem~A~(3). The last
assertion of Theorem~C on the shift-commutation of $\Phi$ follows by
construction.

\section{Applications} \label{sect:applications}
Theorem~C allows to obtain the main results from
\cite{Ringel:Schmidmeier:2008b} and further properties as direct
consequences of properties from the theory of weighted projective
lines. Indeed, as a general rule, we will prove results first for the
category $\vect\XX$ or the stable category $\svect\XX$ of vector
bundles, and then export such results to $\nilop{p}$ or $\snilop{p}$.
In particular, the difficult classification for the tubular case
$\nilop{6}$ thus appears as a consequence of the classification of
indecomposable bundles on $\vect\XX(2,3,6)$ from
\cite{Lenzing:Meltzer:1993} which is analogous to Atiyah's
classification~\cite{Atiyah:elliptic} of vector bundles on a smooth
elliptic curve. Of course, we also use the approach to establish
additional properties of $\snilop{p}$ among them the existence of
various types of tilting objects and establish that the categories are
Calabi-Yau.

\subsection*{Action of the Picard group}

Obviously, the $\LL$-action on $\vect\XX$ by line bundle twist (=
degree shift) induces an $\LL$-action on
$\svect\XX$. By transport of structure, Theorem~C then induces an
$\LL$-action on $\snilop{p}$. This action of the Picard group of $\XX$
on $\svect\XX=\snilop{p}$ reveals a certain amount of symmetry of
$\svect\XX$ which is instrumental in proving most of the properties to
follow. (An important example is the Calabi-Yau property to be
discussed later. By contrast the treatment of the Fuchsian
singularities in  \cite{KST-2}, \cite{Lenzing:Pena:2011} lacks this amount of
symmetry and only yields a finite number of categories which are
fractionally Calabi-Yau.)

\begin{proposition}
The Picard group $\LL=\LL(2,3,p)$ acts on $\snilop{p}$. Let $s$ denote
the automorphism induced by the degree shift of $\mod\ulPp$, then the
generators $\vx_i$ of $\LL$ act as follows on $\snilop{p}$
\begin{enumerate}
 \item[(i)] $\vx_1$ acts as $\tau^3s^3$,
 \item[(ii)] $\vx_2$ acts as $\tau^2s^2$,
 \item[(iii)] $\vx_3$ acts as $s$.
\end{enumerate}
\end{proposition}

The proof immediately follows from the next lemma.

\begin{lemma}
Let $\LL=\LL(2,3,p)$. Then $\LL$ is generated by $\vx_3$ and
$\vom$. Moreover, we have with $\bx_i=\vx_i+\vom$
\begin{enumerate}
\item[(i)] $\bx_1=\bx_2+\bx_3$,
\item[(ii)] $\vx_2=2\bx_3$,
\item[(iii)] $\vx_1=3\bx_3$.
\end{enumerate}
\end{lemma}
\begin{proof}
  (i) We have
  $\bx_2+\bx_3=\vom+(\vom+\vx_2+\vx_3)=\vom+\vx_1=\bx_1$.

  (ii) $2\bx_3=2\vc-2\vx_1-2\vx_2=\vc-2\vx_2=\vx_2$.

  (iii) $3\bx_3=3\vc-3\vx_1-3\vx_2=\vx_1$.
\end{proof}

The next result is not used otherwise in the paper; it follows using ~\cite{Kussin:Lenzing:Meltzer:2010pre}.
\begin{corollary}
By means of the equivalence $\Phi$, the suspension functor $[1]$ of $\svect\XX$ corresponds to the functor $\tau^3s^3$ on  $\snilop{p}$.
\end{corollary}
\begin{proof}
For weight type $(2,p,q)$ it is shown in~\cite{Kussin:Lenzing:Meltzer:2010pre} that the shift with $\vx_1$ serves as the suspension functor for $\svect\XX$.
\end{proof}

\subsection*{Tilting objects and Orlov's trichotomy}
First we establish two tilting objects in $\svect\XX=\snilop{p}$ with
non-isomorphic endomorphism rings. We recall that an object $T$ in a triangulated
category $\Tt$ is a \emph{tilting object} if first it has no self-extensions, i.e.\ $\Hom{}{T}{T[n]}=0$ for each non-zero integer $n$ and, secondly, it generates $\Tt$ homologically, i.e.\ the condition $\Hom{}{T}{X[n]}=0$ for each integer $n$ forces that $X=0$.
\begin{proposition}\label{prop:nakayama}
Assume $U$ is a simple right $\ulPp$-module lying in $\nilop{p}$. Then
$$
\bigoplus_{a=0,\ldots,p-2,\; b=0,1} \tau^{4a+b}s^{3a+b}(U)
$$
is a tilting object in $\snilop{p}$ with endomorphism ring
$A(2(p-1),3)$; see~\eqref{eqn:nakayama}.
\end{proposition}
\begin{proof}
  Let $E$ be an Auslander bundle. With $\bx_i=\vx_i+\vom$ we put
  $$M=\Set{a\bx_1+b\bx_3}{a=0,\ldots, p-2,\; b=0,1}$$ and define $T$
  as the direct sum
  of all $E(\vx)$, with $\vx$ in $M$.
  It is shown in~Theorem~\ref{thm:tilting} that $T$ is a
  tilting object of $\svect\XX$
  with endomorphism ring $\End{}{T}=A(2(p-1),3)$. Transferred to
  $\snilop{p}$ this yields the claim by
  Proposition~\ref{prop:simple:Auslander}.
\end{proof}

Independently, and by different methods the derived equivalence of the algebras
$A(2(p-1),3)$ and $B(2,p-1)$ was shown by S.~Ladkani~\cite{Ladkani}.

To obtain further interesting tilting objects in $\svect\XX$ we need to enlarge the class of Auslander bundles. For $\vx\in M:=\{b\vx_2+c\vx_3\mid b=0,1;\ c=0,\dots,p-2\}$ we define the \emph{extension bundle} $E\langle\vx\rangle$ as the extension term of the unique non-split
exact sequence $0\ra\Oo(\vom)\ra E\langle\vx\rangle\ra\Oo(\vx)\ra 0$. Note for this that $\Ext1{}{\Oo(\vx)}{\Oo(\vom)}=k$. It follows from~\cite{Kussin:Lenzing:Meltzer:2010pre} that each extension bundle $E\langle \vx\rangle$ (with $\vx$ in $M$) is exceptional in $\coh\XX$ and in $\svect\XX$. More is true,
by~\cite{Kussin:Lenzing:Meltzer:2010pre} the system
$T=\bigoplus_{\vx\in M}E\langle\vx\rangle$ is a tilting object in
$\svect\XX$ with $\End{}{T}=B(2,p-1)$, the incidence algebra of the
poset (\ref{eqn:rectangle})
  that is the $2\times (p-1)$-rectangle with all commutativities. Note
  that such
  diagrams appear in singularity theory. By applying Theorem~C we thus
  obtain the following  result

\begin{proposition} \label{prop:rectangle:tilting}
 The category $\snilop{p}=\svect\XX$ with $\XX=\XX(2,3,p)$ has a
 tilting object $T$ whose endomorphism
 ring is the algebra $B(2,p-1)$. In particular, the algebras
 $A(2(p-1),3)$ and $B(2,p-1)$ are derived equivalent.
 \qed
\end{proposition}
By $B'(2,p-1)$ we denote the incidence algebra of the poset (fully
commutative quiver)
\small
$$
\xymatrixcolsep{0.5pc}
\xymatrix@!C=36pt@R18pt{ 1\ar @{->}[r]\ar
  @{->}[d]\ar@{..}[rd] & 2\ar
  @{->}[r]\ar @{->}[d]\ar@{..}[rd] &3\ar[d]\ar[r]&
  \cdots\ar[r]&p-3\ar[d]\ar@{..}[dr] \ar @{->}[r] & p-2\ar @{->}[r]\ar
  @{->}[d]& p-1\\
  1'\ar @{->}[r] & 2'\ar @{->}[r] &3'\ar[r] &\cdots\ar[r]&(p-3)' \ar
  @{->}[r] & (p-2)'
   &  }
$$\normalsize
\sloppy
\begin{corollary}
Let $S$ be a simple $\ulPp$-module belonging to $\snilop{p}$. Then the
right perpendicular category $\rperp{S}$, consisting of all objects
$X$ from $\snilop{p}$ satisfying $\Hom{}{S}{X[n]}=0$ for each integer
$n$, is triangulated with Serre duality. Moreover, the category
$\rperp{S}$ has tilting objects $U$ and $U'$ such that $\End{}{U}\iso
A(2p-3,3)$ and $\End{}{U'}\iso B'(2,p-1)$. In particular, the algebras
$A(2p-3,3)$ and $B'(2,p-1)$ are derived equivalent.
\end{corollary}
\fussy
\begin{proof}
We switch to the category $\svect\XX$, where we have to calculate the
category $\rperp{E}$ for an Auslander bundle $E$. The first claim
follows from Proposition~\ref{prop:nakayama} and its proof. For the
second claim we use Proposition~\ref{prop:rectangle:tilting} and the
fact that $E\langle \vx_2+(p-2)\vx_3\rangle$ is an Auslander bundle.
\end{proof}

Returning to the context of Proposition~\ref{prop:rectangle:tilting},
we want to give the tilting object $\Phi(T)$ of $\snilop{p}$ a more
concrete shape. We briefly point out what the
$\Phi(E\langle\vx\rangle)$ are in the language of the category
$\snilop{p}$. For this let $\sPp^{\mathrm{up}}$, resp.\
$\sPp^{\mathrm{low}}$, denote the full subcategory of $\sPp$ formed by
the objects of the upper (resp.\ lower) bar. Moreover, we identify
$\mmod{\sPp^{\mathrm{up}}}$ with the full subcategory of $\mmod\sPp$
of all modules whose support is contained in
$\sPp^{\mathrm{up}}$. Further we identify $\mmod{\sPp}$ with the
category $\modgr\ZZ{A}$ of finitely generated $\ZZ$-graded modules
over the algebra $A=k[x]/(x^p)$ with $x$ having degree one.
\begin{lemma}
  (a) The restriction functor
  $\rho\colon\mmod\sPp\ra\mmod{\sPp^{\mathrm{up}}}$ has an exact left
  adjoint $\lambda$ sending the indecomposable
  $\sPp^{\mathrm{up}}$-projective $P(\vx)$ to the indecomposable
  $\sPp$-projective $P(\vx_2+\vx)$.

(b) Putting $T^{\mathrm{up}}=\bigoplus_{j=0}^{p-2} E\langle
j\vx_3\rangle$ and $T^{\mathrm{low}}=\bigoplus_{j=0}^{p-2}
E\langle{\vx_2+j\vx_3}\rangle$, the tilting object $T$ from the above
proposition has the form $T=T^{\mathrm{up}}\oplus
T^{\mathrm{low}}$. Moreover, with the above identifications this
yields:
$$
\Phi(T^{\mathrm{up}})=\bigoplus_{j=0}^{p-2}x^{j+1}A(j)\quad\textrm{
  and }\quad
\Phi(T^{\mathrm{low}})=\bigoplus_{j=0}^{p-2}\la(x^{j+1}A(j)).\qed
$$
\end{lemma}
Independently this tilting object in $\snilop{p}$ was constructed by
Xiao-Wu Chen~\cite{Chen} with a direct argument not relying on
Theorem~C.

Since the Grothendieck group $\Knull{\coh\XX}$ is free abelian of rank
$p+4$, see~\cite{Geigle:Lenzing:1987}, we obtain from
Proposition~\ref{prop:nakayama} or
Proposition~\ref{prop:rectangle:tilting} the next result.

\begin{corollary} \label{cor:Knull}
  The Grothendieck group of $\svect\XX(2,3,p)=\snilop{p}$ is free
  abelian of rank
  $2(p-1)$. Moreover, we have
  $\rank\Knull{\snilop{p}}-\rank\Knull{\coh\XX}=p-6$. \qed
\end{corollary}
This result serves as a nice illustration of an $\LL$-graded version of Orlov's
theorem~\cite{Orlov:2009}. For this we recall from the Introduction,
(\ref{eq:threeCY}) that there are natural equivalences
\begin{equation}\label{eq:fourCY}
\Tt:=\Dsing\LL{S}=\sCMgr\LL{S}=\svect\XX=\snilop{p}, \textrm{ where
}\XX=\XX(2,3,p)
\end{equation}
where we have used Theorem~C for the last identification.
It follows from an $\LL$-graded version of Orlov's
theorem~\cite{Orlov:2009} that the comparison between $\Der{\coh\XX}$
and any of the four triangulated categories above follows a trichotomy
determined by the \emph{Gorenstein parameter} of the singularity. In
the present $\LL$-graded setting this index equals $6-p$, compare
Corollary~\ref{cor:Knull}. (We have normalized the sign in order to
make it equal to the sign of the Euler characteristic.)
\begin{proposition}[Orlov's trichotomy]
Let $\XX=\XX(2,3,p)$. Then the categories $\Der{\coh\XX}$ and $\Tt$
are related as follows.
\begin{enumerate}
\item[(1)] For $\eulerchar\XX>0$ the category $\Tt$ is
  triangle-equivalent to the right perpendicular category in
  $\Der{\coh\XX}$ with respect to an exceptional sequence of $6-p$
  members;
\item[(2)] For $\eulerchar\XX=0$ the category $\Tt$ is
  triangle-equivalent to $\Der{\coh\XX}$;
\item[(3)] For $\eulerchar\XX<0$ the category $\Der{\coh\XX}$ is
  triangle equivalent to the right perpendicular category in $\Tt$
  with respect to an exceptional sequence of $p-6$ members.
\end{enumerate}
\end{proposition}

\subsection*{Calabi-Yau dimension and Euler characteristic} Let $\Tt$
be a triangulated category with Serre duality. Let $S$ denote the
Serre functor of $\Tt$. Assume the existence of a smallest integer
$n\geq1$ such that we have an isomorphism $S^n\iso[m]$ of functors for
some integer $m$. (Here, $[m]$ denotes the $m$-fold suspension of
$\Tt$.) Then $\Tt$ is called \emph{Calabi-Yau} of fractional
CY-dimension $\frac{m}{n}$. Note that the ``fraction'' $\frac{m}{n}$
is kept in uncanceled format.
The bounded derived category $\Der{\coh\XX}$ of coherent sheaves on
$\XX(2,3,p)$ is almost never Calabi-Yau, the only exception being the
tubular case $p=6$, where we have fractional CY-dimension $6/6$. It is
therefore remarkable that the category $\snilop{p}=\svect\XX(2,3,p)$
is always fractional Calabi-Yau. Moreover, the CY-dimension only
depends on the Euler characteristic of $\XX$. To show this the next
lemma is useful.

\begin{lemma} \label{lemma:vom-mod-vx_1}
  Let $\LL=\LL(2,3,p)$. The class of $\vom$ in $\LL/\ZZ\vx_1$ is of
  order $\lcm(3,p)$. Moreover, the equality
  $$\lcm(3,p)\cdot\vom=\Bigl(\lcm(3,p)\cdot\frac{p-6}{3p}\Bigr)\cdot\vx_1$$
  holds in $\LL$.
\end{lemma}
\begin{proof}
  Write $n\geq 1$ as $n=a\cdot p+b$ with $a,\,b\in\ZZ$ and $0\leq
  b<p$. Then we have
  $n\vom=n(\vx_1-\vx_2-\vx_3)=n\vx_1-n\vx_2-n\vx_3$, hence
  $n\vom\in\ZZ\vx_1$ if and only if $3\mid n$ and $p\mid n$, which is
  equivalent to $\lcm(3,p)\mid n$. This shows the first claim. The
  second follows from
  $\lcm(3,p)\cdot\vom=\lcm(3,p)\cdot\Bigl(1-\ffrac{2}{3}-\ffrac{2}{p}\Bigr)
  \cdot\vx_1$.
\end{proof}

\begin{proposition}
  The category $\snilop{p}$ is Calabi-Yau of fractional Calabi-Yau
  dimension $d_p$ given
  as follows:
  \begin{eqnarray*}
    d_2 & = & \frac{1}{3}\ \ (=1-2\cdot\eulerchar{\XX}),\\
    d_p & = & \frac{\lcm(3,p)\cdot
    (1-2\cdot\eulerchar{\XX})}{\lcm(3,p)},\ \text{for}\ p\geq 3.
  \end{eqnarray*}
  Here, $\eulerchar{\XX}=\ffrac{1}{p}-\ffrac{1}{6}=\ffrac{(6-p)}{6p}$ is
  the Euler characteristic of $\XX(2,3,p)$, and
  $1-2\cdot\eulerchar{\XX}=\ffrac{(4p-6)}{3p}$.
\end{proposition}
Note that the nominator of $d_p$ is always an integer.
\begin{proof}
  Assume first that $p\geq 3$. Then the Picard group $\LL=\LL(2,3,p)$
  acts faithfully on $\svect{\XX}$. Indeed, if $E$ is an Auslander
  bundle with $E(\vx)\simeq E$ in $\svect{\XX}$, then $p\geq 3$
  implies $\vx=0$. (In case $p\geq 3$ inspection of the AR components
  shows that for two line bundles $L,\,L'$ the corresponding Auslander
  bundles $E(L),\,E(L')$ are isomorphic if and only if $L$ and $L'$
  are.) Since shift by $\vx_1$ serves as suspension $[1]$ and the
  Serre functor on $\svect{\XX}$ is given by $S=\tau[1]$, it follows
  from the preceding lemma that the fractional Calabi-Yau dimension of
  $\svect{\XX}$ is given by $\frac{m+n}{n}$ with $n=\lcm(3,p)$ and
  $m=\lcm(3,p)\cdot\frac{p-6}{3p}$. For $p=2$ we have a similar
  formula, but in the resulting fraction $2/6$ the factor $2$ can be
  canceled, since in this case for two line bundles $L,\,L'$ the
  corresponding Auslander bundles $E(L),\,E(L')$ are isomorphic if and
  only if $L'\simeq L$ or $L'\simeq L(\vx_1 -\vx_3)$.
\end{proof}

\begin{corollary}
  The category $\snilop{p}$ determines $\coh\XX$. \qed
\end{corollary}

For distinction, we denote the class of a vector bundle $X$ in $\Knull{\svect\XX}$ by $[[X]]$, and use $\seuler{[[X]]}{[[Y]]}=\sum_{n\in\ZZ}\dim{\sHom{}{X}{Y[n]}}$ for the \emph{Euler form}.

\begin{lemma} \label{lemma:stable:eulerform}
We assume weight type $(2,3,p)$ with $p\geq3$. Let $E$ be an Auslander bundle. Let $h=\lcm(6,p)$ and $A=\set{0,\bx_1,\bx_2,\bx_3}$. Then the following holds:

(i) We have $\seuler{[[E]]}{[[E(j\vom)]]}=1$ if and only if $j\vom\in A$ modulo $\ZZ\vc$.

(ii) We have $\seuler{[[E]]}{[[E(j\vom)]]}=-1$ if and only if $j\vom-\vx_1\in A$ modulo $\ZZ\vc$.

(iii) Assume $1\leq j < h$, then $[[E(j\vom)]]\neq[[E]]$.

(iv) Assume $p\geq4$ is even. Then $[[E(\frac{h}{2}\vom)]]\neq-[[E]]$.
\end{lemma}

\begin{proof}
Concerning (i) assume first that $j\vom$ is congruent to $\vy\in A$ modulo $\ZZ\vc$. Then $[[E]]=[[E(\vy)]]$ and $\seuler{[[E]]}{[[E\vy]]}=1$ by Proposition~\ref{prop:stable_morphisms}. Conversely, assume $\seuler{[[E]]}{[[E(j\vom)]]}>0$. By Proposition~\ref{prop:stable_morphisms} there exists an integer $n_0$ such that $\vy=j\vom+n_0\vx_1$ belongs to $A$ and, moreover, $j\vom +n\vx_1\notin A$ for all integers $n\neq n_0$. Since $E(n\vx_1)=E[n]$, we obtain $\seuler{[[E]]}{[[E(j\vom)]]}=\seuler{[[E]]}{[[E(\vy)]]}=1$.
The proof of (ii) is similar to the previous one. Passing to claim (iii), assume $[[E(j\vom)]]=[[E]]$ and then $\seuler{[[E]]}{[[E(j\vom)]]}=\seuler{[[E]]}{[[E]]}=1$. By (i) we obtain that $j\vom\in A$ modulo $\ZZ\vc$. By the assumption on $j$ this excludes the possibility $j\vom=0$ modulo $\ZZ\vc$, hence $j\vom=\bx_i+n\vc$ for some $i=1,2,3$ and $n\in\ZZ$. Since up to a common degree shift $E$ and $E(\bx_i)$ belong to the tilting object from Theorem~\ref{thm:tilting} we conclude that  $[[E(\bx_i)]]\neq[[E]]$, thus contradicting our assumption.
\end{proof}

Recall that the \emph{Coxeter transformation} of a triangulated
category $\Tt$ with Serre duality is the automorphism of the
Grothendieck group of $\Tt$ induced by the Auslander-Reiten
translation $\tau=S[-1]$, where $S$ denotes the Serre functor for
$\Tt$.
From Lemma~\ref{lemma:vom-mod-vx_1} we then deduce the following,
see~\cite{Kussin:Lenzing:Meltzer:2010pre}.

\begin{proposition}
  The Coxeter transformation $\phi$ of $\snilop{p}=\svect\XX(2,3,p)$
  has order $h=3$
  for $p=2$ and order $h=\lcm(6,p)$ otherwise. Moreover, assuming
  $p\geq3$ we have
  $\phi^{h/2}=-1$ if and only if $p$ is odd. \qed
\end{proposition}
Note that, in classical situations, $h$ is called the \emph{Coxeter
  number}, a nomination which we extend to the present context.

\begin{proof}
Assume $p\geq3$ such that $h=\lcm(6,p)$. We have $h\vom=\delta(\vom)\vc$ such that $\tau^h=[2\delta(\vom)]\vc$. Passing to the Grothendieck group of $\svect\XX$ we obtain $\phi^h=1$. By Theorem~\ref{thm:tilting} there is a system of Auslander bundles whose classes form a $\ZZ$-basis of $\Knull{\svect\XX}$. Then Lemma~\ref{lemma:stable:eulerform} (iii) implies that $h$ is the precise order of $\phi$. Assuming $p$ odd, then Lemma~\ref{lemma:vom-mod-vx_1} implies that $\phi^{h/2}=-1$. Moreover, Lemma~\ref{lemma:stable:eulerform}~(iv) shows that this is not the case for $p$ even. \end{proof}

\subsection*{Shape of the categories $\nilop{p}$ and $\snilop{p}$}

In this subsection we show how the structural results
of~\cite{Ringel:Schmidmeier:2008b} for $\nilop{p}$ and $\snilop{p}$,
in particular the assertions of the shape of Auslander-Reiten
components, follow
from Theorem~C. The shape of the results depends sensibly on the
(orbifold) Euler characteristic  $\chi_\XX=1/p-1/6$ of
$\XX(2,3,p)$. Note that $\chi_\XX$ is $>0, =0, <0$ if
and only if $p<6$, $p=6$ or $p>6$, respectively.

The Auslander-Reiten components of $\vect\XX/[\Ff]$ or of
$\svect\XX=\vect\XX/[\Ll]$ are those from $\vect\XX$ with all line
bundles from $\Ff$ (resp.\ $\Ll$) removed. By transport of structure
this allows to determine the Auslander-Reiten structure of $\nilop{p}$
and $\snilop{p}$, thus obtaining the corresponding results of
\cite{Ringel:Schmidmeier:2008b}. We remark that it may be deduced from Proposition~\ref{prop:simple:Auslander} under the functor $\Phi$ the Auslander bundles correspond exactly to the boundary modules from \cite[Section 5.1]{Ringel:Schmidmeier:2008b}.

\subsubsection*{Fundamental domain under shift}
It is shown in~\cite{Ringel:Schmidmeier:2008b} that identification
$E=E(\vx_3)$ yields for $p\leq 5$ the ungraded invariant subspace
problem $S(p)$. More explicitly, with $\XX=\XX(2,3,p)$ we have a
covering functor $\vect\XX/[\Ff]=\nilop{p}\ra \Ss(p)$ with infinite
cyclic covering group $G$ generated by the degree shift
$\si_3=\si(\vx_3)$ with $\vx_3$. In the next
proposition we describe explicitly a fundamental domain in
$\vect\XX/[\Ff]$ with respect to this $G$-action. From
\cite{Ringel:Schmidmeier:2008b} one obtains a full embedding of the
orbit category  $(\vect\XX/[\Ff])/G\hookrightarrow\nilop{p}$. It is
shown in~\cite{Ringel:Schmidmeier:2008b} that for $p\leq6$ this
embedding is actually an equivalence. It is conjectured that for
$p\geq 7$  the above embedding is not dense.

To describe a fundamental domain with respect to the $G$-action, we
recall from \cite{Geigle:Lenzing:1987} that the slope of a vector
bundle $X$ is defined by $\mu X=\deg{X}/\rank{X}$, where the degree
deg is the linear form on $\Knull{\coh\XX}$ which is uniquely determined by
$\deg{\Oo(\vx)}=\de(\vx)$ and where $\de\colon\LL\ra\ZZ$ is the
homomorphism sending $\vx_1$, $\vx_2$, $\vx_3$ to $\lcm(6,p)/2$,
$\lcm(6,p)/3$, $\lcm(6,p)/p$, respectively.

\begin{proposition} \label{prop:fundamental:domain} Let $\XX$ be of
  weight type $(2,3,p)$, $p\geq3$. Then the following holds:
\begin{enumerate}
 \item[(i)] The indecomposable bundles $X$ not in $\Ff$ having slope
 in the range $0\leq \mu X< \de(\vx_3)$ form a fundamental domain
 $\Dd$ of $\vect\XX/[\Ff]$ with respect to the
 $\langle\si_3\rangle$-action.
 \item[(ii)] There are exactly $6$ line bundles $L$ with slope in the
 range $0\leq \slope{L}<\de(\vx_3)$.
 \item[(iii)] $\Dd$ contains exactly two (persistent) line bundles,
 one of them from the upper bar the other one from the lower bar.
 \item[(iv)] $\Dd$ contains exactly $6$ Auslander bundles.
\end{enumerate}
\begin{proof}
  Assertion (i) follows from the formula
  $\mu(E(\vx_3))=\mu\,E+\de(\vx_3)$. For (ii) we recall that
  $\ZZ\vx_3$ has index $6$ in $\LL$. Moreover, each $\langle
  \si_3\rangle$-orbit $\Set{L(n\vx_3)}{n\in\ZZ}$ contains exactly one
  line bundle in the given slope range. Assertion (iii) amounts to
  determine all $\vx$ of shape $n\vx_3$ or $\vx_2+n\vx_3$ satisfying
  $0\leq\de(\vx)<\de(\vx_3)$. Claim (iv) is a direct consequence of
  (ii).
\end{proof}
\end{proposition}

\subsubsection*{Positive Euler characteristic}
This deals with the cases $p=2,3,4$ and $5$. Note that the treatment
is related to \cite{KST-1}, but except for $p=5$ deals with a
different situation.
\begin{proposition}
For $2\leq p\leq 5$ let $\De=[2,3,p]$ (resp.\ $\tilde\De$) be the
attached Dynkin (resp.\ extended Dynkin) diagram.
\begin{enumerate}
\item[(1)]The Auslander-Reiten quiver of $\vect\XX$ consists of a
  single standard component. The category of indecomposable vector
  bundles on $\XX$ is equivalent to the mesh category of the
  Auslander-Reiten component $\ZZ\tilde\Delta$. (The vertices
  corresponding to persistent (fading) vector bundles will be called
  persistent (fading).)
\item[(2)]The Auslander-Reiten quiver $\Ga$ of
  $\vect\XX/[\Ff]=\nilop{p}$ consists of
  a single component. It is obtained from the translation quiver
  $\ZZ\tilde\Delta$ by deleting the fading vertices and adjacent
  arrows. The category of indecomposable objects of $\nilop{p}$ is
  equivalent to the mesh-category of $\Ga$.
\item[(3)]  The category $\svect\XX=\snilop{p}$ is equivalent to
  $\Der{\mmod{K\vec{\De_p}}}$ for some quiver $\vec\De_p$ with
  underlying Dynkin graph $\De_2=\AA_2$, $\De_3=\DD_4$, $\De_4=\EE_6$
  and $\De_5=\EE_8$.
  \end{enumerate}
\end{proposition}
\begin{proof}
We only sketch the argument, for further details we refer
to~\cite{Kussin:Lenzing:Meltzer:2010pre}. One first shows that the
direct sum of all indecomposable bundles $E$ with slope in the range
$0\leq \slope{E}< -\de(\vom)$ yields a tilting object $T$ for
$\coh\XX$. This allows to prove assertion (1). Assertion (2) then
follows from (1) using Theorem~C. For (3) we use that the
indecomposable summands of $T$ which are not line bundles yield a
tilting object $T'$ for $\svect\XX$ whose endomorphism ring is as
described in (3).
\end{proof}
By way of example we treat the cases $\nilop{4}$ and $\nilop{5}$. In
Figure~\ref{figure:234} we illustrate a fundamental domain in the
Auslander-Reiten quiver of $\nilop{4}$ modulo the shift action by
$\mathbb{Z}\vx_3$. The line bundles are the objects at the upper and
lower boundary of the graph. In the following figures the fading line bundles
are indicated by circles and the adjacent (fading) arrows are
dotted. All other objects (in particular the persistent line bundles)
are marked by fat points.
\begin{figure}[H]
$$
\def\c{\circ} \def\b{\bullet}
\xymatrix@R8pt@C8pt{ L' & \c\ar@{.>}[rd]
  \ar @{-}[dddddd] & &\b\ar[rd] &
  &\c\ar@{.>}[rd]  &                       &\b \ar @{-}[dddddd] & L(\vx_3)\\
  & &\b\ar[ru]\ar[rd] & &*+[F]{\b}\ar@{.>}[ru]\ar[rd] & &\b\ar[ru]\ar[rd] &
  \\
  & \b\ar[ru]\ar[rd]& &*+[F]{\b}\ar[ru]\ar[rd]&
  &\b\ar[ru]\ar[rd]&                       &\b & \\
  & \b\ar[r] &\b\ar[ru]\ar[rd]\ar[r]&*+[F]{\b}\ar[r]
  &*+[F]{\b}\ar[ru]\ar[rd]\ar[r]&\b\ar[r]
  &\b\ar[ru]\ar[rd]\ar[r]&\b\\
  & \b\ar[ru]\ar[rd]& &*+[F]{\b}\ar[ru]\ar[rd]&
  &\b\ar[ru]\ar[rd]&                       &\b & \\
  & &\b\ar[ru]\ar@{.>}[rd] & &*+[F]{\b}\ar[ru]\ar@{.>}[rd] & &\b\ar[ru]\ar@{.>}[rd] &
  & \\
  L & \b\ar[ru] & &\c\ar@{.>}[ru] & &\c\ar@{.>}[ru] & &\c & L' (\vx_3)}
$$
\caption{Fundamental domain for $\snilop{4}$}\label{figure:234}
\end{figure}
\noindent We have marked the indecomposable summands of a tilting object for $\snilop{4}$ with endomorphism ring of Dynkin type $\EE_6$.

In Figure~\ref{figure:235} we illustrate a fundamental domain in the
Auslander-Reiten quiver of $\nilop{5}$ modulo the shift action by
$\mathbb{Z}\vx_3$. Here the line bundles are the objects at the lower
boundary of the quiver.
\begin{figure}[H]
$$
\def\c{\circ}
\def\b{\bullet}
\xymatrix@R7pt@C7pt{
& \ar @{-}[ddddddd] &\b\ar[rd] &                       &\b\ar[rd]       &
&\b\ar[rd]       &                       &\b\ar[rd]       &
& *+[F]{\b}\ar[rd]       &                       &\b\ar[rd]       & \ar
@{-}[ddddddd] & \\
&
\b\ar[ru]\ar[rd]&&\b\ar[ru]\ar[rd]&&\b\ar[ru]\ar[rd]&&\b\ar[ru]\ar[rd]&&*+[F]{\b}\ar[ru]\ar[rd]&&\b\ar[ru]\ar[rd]&&\b
&\\
& \b\ar[r] &\b\ar[ru]\ar[rd]\ar[r]  &\b\ar[r]
&\b\ar[ru]\ar[rd]\ar[r] &\b\ar[r]
&\b\ar[ru]\ar[rd]\ar[r]&\b\ar[r] &\b\ar[ru]\ar[rd]\ar[r]&*+[F]{\b}\ar[r]
&*+[F]{\b}\ar[ru]\ar[rd]\ar[r] &\b\ar[r]         &\b\ar[ru]\ar[rd]\ar[r]
&\b &\\
& \b\ar[ru]\ar[rd]       &                &\b\ar[ru]\ar[rd]       &
&\b\ar[ru]\ar[rd]       &                &\b\ar[ru]\ar[rd]       &
&*+[F]{\b}\ar[ru]\ar[rd]&                &\b\ar[ru]\ar[rd]       &
&\b &\\
 & & \b\ar[ru]\ar[rd]&                       &\b\ar[ru]\ar[rd]&
 &\b\ar[ru]\ar[rd]&                       &\b\ar[ru]\ar[rd]&
 & *+[F]{\b}\ar[ru]\ar[rd]&                       &\b\ar[ru]\ar[rd]& &\\
& \b\ar[ru]\ar[rd]       &                &\b\ar[ru]\ar[rd]       &
&\b\ar[ru]\ar[rd]       &                &\b\ar[ru]\ar[rd]       &
&*+[F]{\b}\ar[ru]\ar[rd]&                &\b\ar[ru]\ar[rd]       &
&\b &\\
 & &\b\ar[ru]\ar@{.>}[rd]&                       &\b\ar[ru]\ar@{.>}[rd]&
 &\b\ar[ru]\ar@{.>}[rd]&                       &\b\ar[ru]\ar[rd]&
 &*+[F]{\b}\ar[ru]\ar@{.>}[rd]&                       &\b\ar[ru]\ar[rd]& &\\
L & \b\ar[ru]& &\c\ar@{.>}[ru]       &                       &\c\ar@{.>}[ru]
 &                       &\c\ar@{.>}[ru]       &
 &\b\ar[ru]       &                       &\c\ar@{.>}[ru]       &
 &\b & L(\vx_3)}
$$
\caption{Fundamental domain for $\nilop{5}$}\label{figure:235}
\end{figure}
\noindent We have marked a tilting object for the stable category
$\svect{5}=\snilop{5}$ with endomorphism ring of Dynkin type $\EE_8$.
\subsubsection*{Euler characteristic zero, the case {$p=6$}}
For $\chi_\XX=0$, that is $p=6$, the category $\coh\XX$ is tubular of
type $(2,3,6)$. Hence the line bundles are exactly the objects in the
tubes of integral slope and of $\tau$-period $6$, see
\cite{Lenzing:Meltzer:1993}.  Passing to the factor category
$\vect\XX/[\Ff]=\nilop{p}$ all other Auslander-Reiten components
remain unchanged, while the ``line bundle components'' get the shape
from Figure~\ref{figure:tube}
\begin{figure}[H]
$$
\def\c{\circ} \def\b{\bullet} \xymatrix@-1pc@!R=8pt@!C=8pt{ & \vdots & \vdots
  & \vdots  & \vdots
  & \vdots & \vdots & \vdots & \vdots & \vdots
  & \vdots  & \vdots & \\
  & \b\ar[rd]       &                       &\b\ar[rd]       &
  &\b\ar[rd]       &                       &\b\ar[rd]       &
  & \b\ar[rd]       &                       &\b\ar[rd]       & \\
  \b\ar[ru]\ar[rd]       &                &\b\ar[ru]\ar[rd]       &
  &\b\ar[ru]\ar[rd]       &                &\b\ar[ru]\ar[rd]       &
  &\b\ar[ru]\ar[rd]&      &\b\ar[ru]\ar[rd]       &
  &\b\\
  &\b\ar[ru]\ar[rd]&                       &\b\ar[ru]\ar[rd]&
  &\b\ar[ru]\ar[rd]&                       &\b\ar[ru]\ar[rd]&
  &\b\ar[ru]\ar[rd]&                       &\b\ar[ru]\ar[rd]& \\
  \b\ar[ru]\ar[rd]       &                &\b\ar[ru]\ar[rd]       &
  &\b\ar[ru]\ar[rd]       &                &\b\ar[ru]\ar[rd]       &
  &\b\ar[ru]\ar[rd]&                &\b\ar[ru]\ar[rd]       &
  &\b\\
  & \b\ar[ru]\ar[rd]&                       &\b\ar[ru]\ar[rd]&
  &\b\ar[ru]\ar[rd]&                       &\b\ar[ru]\ar[rd]&
  & \b\ar[ru]\ar[rd]&                       &\b\ar[ru]\ar[rd]& \\
  \b\ar[ru]\ar[rd]&                       &\b\ar[ru]\ar[rd]&
  &\b\ar[ru]\ar[rd]&                       &\b\ar[ru]\ar[rd]&
  &\b\ar[ru]\ar[rd]&                       &\b\ar[ru]\ar[rd]&
  &\b\\
  &\b\ar[ru]\ar@{.>}[rd]&                       &\b\ar[ru]\ar@{.>}[rd]&
  &\b\ar[ru]\ar@{.>}[rd]&                       &\b\ar[ru]\ar[rd]&
  &\b\ar[ru]\ar@{.>}[rd]&                       &\b\ar[ru]\ar[rd]& \\
  \b\ar[ru] \ar @{-}[uuuuuuuu]& &\c\ar@{.>}[ru] & &\c\ar@{.>}[ru] &
  &\c\ar@{.>}[ru] &
  &\b\ar[ru] & &\c\ar@{.>}[ru] & &\b \ar @{-}[uuuuuuuu]}
$$
\caption{Tube for $\nilop{6}$ containing line bundles.}\label{figure:tube}
\end{figure}
\noindent (line bundles at the lower boundary). Here the
identification yields standard (non-stable) tubes of $\tau$-period
$6$. Concerning the stable category $\svect\XX=\snilop{6}$, we have
the following result.
\begin{proposition}\label{prop:tubular} Assume $\XX$ has weight type
  $(2,3,6)$. Then there exists a tilting object in the stable category
  $\svect\XX=\snilop{6}$ whose endomorphism ring is the canonical
  algebra $\La=\La(2,3,6)$. In particular, we have triangle
  equivalences $\svect\XX=\snilop{6}\iso\Der{\coh\XX}$.
\end{proposition}
\begin{proof}
We sketch the argument, leaving details
to~\cite{Kussin:Lenzing:Meltzer:2010pre}. As shown
in~\cite{Geigle:Lenzing:1987}, the direct sum $T$ of all line bundles
$\Oo(\vx_3+\vx)$ with $\vx$ in the range $0\leq\vx\leq\vc$ is a
tilting object for $\coh\XX$ and $\Der{\coh\XX}$. By
\cite{Lenzing:Meltzer:1993} there is an auto-equivalence $\rho$ of
$\Der{\coh\XX}$ acting on slopes $q$ by $q\mapsto 1/(1+q)$. It follows
that $\rho T$ is a bundle whose indecomposable summands have slopes
$q$ in the range $1/2< q <1$. It follows from this property that $\rho
T$ is a tilting object for $\svect\XX$ having all the claimed
properties.
\end{proof}
Recall in this context that the category $\Hh=\coh\XX$ is hereditary,
yielding the very concrete description of $\Der{\coh\XX}$ as the
\emph{repetitive category} $\bigvee_{n\in\ZZ}\Hh[n]$, where each
$\Hh[n]$ is a copy of $\Hh$ (objects written $X[n]$ with $X\in\Hh$)
and where morphisms are given by
$\Hom{}{X[n]}{Y[m]}=\Ext{m-n}\Hh{X}{Y}$ and composition is given by
the Yoneda product.
\begin{remark}
The classification of indecomposable bundles over the weighted
projective line $\XX=\XX(2,3,6)$ is very similar to Atiyah's
classification of vector bundles on a smooth elliptic curve, compare
\cite{Atiyah:elliptic} and \cite{Lenzing:Meltzer:1993}. Indeed the
relationship is very close: Assume the base field is algebraically
closed of characteristic different from $2$ and $3$. If $E$ is a
smooth elliptic curve of $j$-invariant $0$, it admits an action of the
cyclic group $G$ of order $6$ such that the category
$\mathrm{coh}_G(E)$ of $G$-equivariant coherent sheaves on $E$ is
equivalent to $\coh\XX$. Thus $\snilop{6}$ has the additional
description as stable category
$\underline{\mathrm{vect}}_G\textrm{-}E$ of $G$-equivariant vector
bundles on $E$.
\end{remark}

\subsubsection*{Negative Euler characteristic}
Let $\chi_\XX<0$, accordingly $p\geq7$. Here, the classification
problem for $\svect\XX=\snilop{p}$ is wild. The study of these
categories is related to the investigation of Fuchsian singularities
in  \cite{KST-2, Lenzing:Pena:2011} but,  with the single exception
$p=7$ yields a different stable category of vector bundles (since only
one $\tau$-orbit of line bundles is factored out in the Fuchsian
case).

It is shown in~\cite{Lenzing:Pena:1997} that, for $\XX=\XX(2,3,p)$ and
$p\geq7$, all
Auslander-Reiten components for $\vect\XX$ have the shape $\ZZ
A_\infty$, and we have exactly $|\LL/\ZZ\vom|=p-6$ components
containing a line bundle. Only the shape of these components is affected when
passing to the factor category $\vect\XX/[\Ff]=\nilop{p}$,

\begin{proposition}
 For $p\geq 7$ each Auslander-Reiten component of
  $\svect\XX(2,3,p)=\snilop{p}$ is of shape
  $\ZZ\AA_\infty$. Moreover there is a natural bijection between the
  set of all Auslander-Reiten components to the set of all regular
  Auslander-Reiten components
  over the wild path algebra $\La_0$ over the star $[2,3,p]$.
\end{proposition}
\begin{proof}
Invoking stability arguments, all line bundles lie at the border of
their Auslander-Reiten component in
$\vect\XX$~\cite{Lenzing:Pena:1997}. Passage to the stable category
then shows that all components in $\svect\XX$ have shape
$\ZZ\AA_\infty$. The argument implies, moreover, that there is a
natural bijection between the set of Auslander-Reiten components in
$\vect\XX$ and in $\svect\XX$, respectively. The claim then follows
from \cite{Lenzing:Pena:1997}.
\end{proof}
\begin{figure}[H]
\small
$$
\def\c{\circ} \def\b{\bullet} \xymatrix@-1pc@!R=7pt@!C=7pt{& & & \vdots &
  \vdots & \vdots & \vdots & \vdots & \vdots & \vdots & \vdots &
  \vdots
  & \vdots  & \vdots & &\\
  && & \b\ar[rd] & &\b\ar[rd] & &\b\ar[rd] & &\b\ar[rd] &
  & \b\ar[rd]       &                       &\b\ar[rd]       & & \\
  && \b\ar[ru]\ar[rd] & &\b\ar[ru]\ar[rd] & &\b\ar[ru]\ar[rd] &
  &\b\ar[ru]\ar[rd] & &\b\ar[ru]\ar[rd]& &\b\ar[ru]\ar[rd] & &\b
  &\\
  && &\b\ar[ru]\ar[rd]&                       &\b\ar[ru]\ar[rd]&
  &\b\ar[ru]\ar[rd]&                       &\b\ar[ru]\ar[rd]&
  &\b\ar[ru]\ar[rd]&                       &\b\ar[ru]\ar[rd]& &\\
  && \b\ar[ru]\ar[rd] & &\b\ar[ru]\ar[rd] & &\b\ar[ru]\ar[rd] &
  &\b\ar[ru]\ar[rd] & &\b\ar[ru]\ar[rd]& &\b\ar[ru]\ar[rd] &
  &\b &\\
  && & \b\ar[ru]\ar[rd]&                       &\b\ar[ru]\ar[rd]&
  &\b\ar[ru]\ar[rd]&                       &\b\ar[ru]\ar[rd]&
  & \b\ar[ru]\ar[rd]&                       &\b\ar[ru]\ar[rd]& &\\
  && \b\ar[ru]\ar[rd]& &\b\ar[ru]\ar[rd]& &\b\ar[ru]\ar[rd]&
  &\b\ar[ru]\ar[rd]& &\b\ar[ru]\ar[rd]& &\b\ar[ru]\ar[rd]&
  &\b &\\
  && &\b\ar[ru]\ar@{.>}[rd]& &\b\ar[ru]\ar@{.>}[rd]& &\b\ar[ru]\ar@{.>}[rd]&
  &\b\ar[ru]\ar[rd]&
  &\b\ar[ru]\ar@{.>}[rd]&                       &\b\ar[ru]\ar[rd]& & \\
  &L & \b\ar[ru] \ar @{-}[uuuuuuuu]& &\c\ar@{.>}[ru] &
  &\c\ar@{.>}[ru] & &\c\ar@{.>}[ru]
  & &\b\ar[ru] & &\c\ar@{.>}[ru] & &\b \ar @{-}[uuuuuuuu] && L((p-6)\vx_3)}
$$
\caption{Case $p\geq7$. Fundamental domain for the ``distinguished''
  components}\label{figure:wild}
\end{figure}
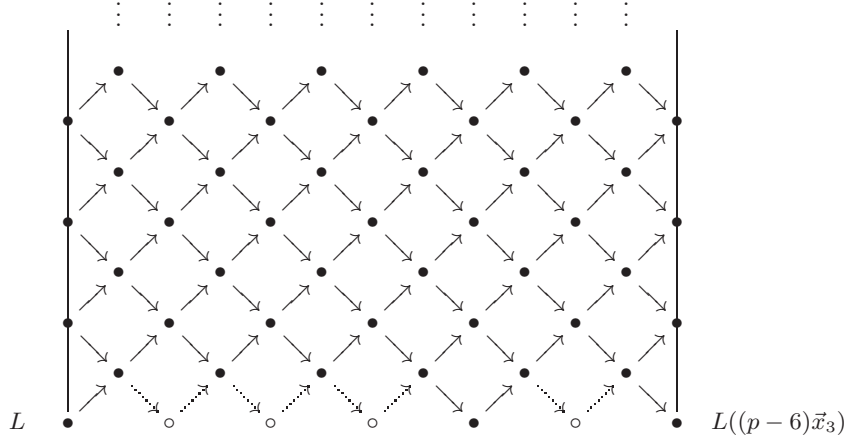
\normalsize
\noindent The picture is a nice illustration for
Proposition~\ref{prop:fundamental:domain}.
For $p\geq7$ the class of $\vx_3$ is a generator of  $\LL/\ZZ\vom$
having  order $p-6$. Accordingly shift with $\vx_3$ acts on the
$(p-6)$-element set of ``distinguished'' components by cyclic
permutation. Figure~\ref{figure:wild} therefore shows a fundamental
domain $\Dd$ for the $p-6$ ``distinguished'' components.

\subsubsection*{ADE-chain}
The table below summarizes previous results and displays for
$\svect\XX=\snilop{p}$ with $\XX$ of type $(2,3,p)$ the fractional
Calabi-Yau dimension, the Euler characteristic $\chi_\XX$, the Coxeter
number $h$, the representation type, and the
derived type of $\snilop{p}$ for small values of $p$.
\small
\begin{table}[H]
\begin{displaymath}
\def\l{\langle} \def\r{\rangle}
\renewcommand{\arraystretch}{1.4}
\arraycolsep0pt
\begin{tabular}{|c||cccc|c|cccc|} \hline
  $p$ & 2 & 3 & 4 & 5 & 6 & 7 & 8 & 9&$p$\\
\hline
CY-dim&$\frac{1}{3}$
&$\frac{2}{3}$&$\frac{10}{12}$&$\frac{14}{15}$&$\frac{6}{6}$&$\frac{22}{21}$&$\frac{26}{24}$&
$\frac{10}{9}$&$\frac{\lcm(3,p)\cdot
    (1-2\cdot\eulerchar{\XX})}{\lcm(3,p)}$\\\hline
$\chi_\XX$&$\frac{1}{3}$
&$\frac{1}{6}$&$\frac{1}{12}$&$\frac{1}{30}$&$0$
&$-\frac{1}{42}$&$-\frac{1}{24}$&
$-\frac{1}{18}$&$\frac{1}{p}-\frac{1}{6}$\\\hline
$h$&3&6&24&30&6&42&24&18&$\lcm(6,p)$\\\hline
type      &$\AA_2$       &$\DD_4$      &$\EE_6$       &$\EE_8$
&$(2,3,6)$     &$\l 2,3,7\r$  &
$\l 2,3,8\r$  &$\l 2,3,9\r$&$\l 2,3,p\r$ \\ \hline
repr.\ type&\multicolumn{4}{c|}{repr.-finite}& tubular&
\multicolumn{4}{c|}{wild, new type}\\\hline
\end{tabular}
\end{displaymath}
\caption{An ADE-chain}
\end{table}
\normalsize
\noindent The table expresses an interesting property of the sequence
of triangulated categories $\svect\XX(2,3,p)=\snilop{p}$. For small
values of $p$, the category $\svect\XX=\snilop{p}$ yields Dynkin
type. For $p=6$ the sequence passes the `borderline' of tubular type
and then continues with wild type. While such situations occur
frequently,  it is quite rare that one gets an infinite sequence of
categories $\Tt_n$ which all are fractional Calabi-Yau and where the
size of $\Tt_n$, measured in terms of the Grothendieck group, is
increasing with $n$.

Returning to the particular case $\Tt_p=\svect\XX(2,3,p)$ we know from Theorem~\ref{thm:tilting} that
$\Tt_p$ has a tilting object $T$ consisting of the Auslander bundles
$E\ra E(\bx_3)\ra E(\bx_1)\ra \cdots \ra E((p-2)\bx_1)\ra
E((p-2)\bx_1+\bx_3)$ and whose endomorphism ring is $A(n,3)$ with
$n=2(p-1)$. This implies that the right perpendicular category formed
in  $\Tt_p$ with respect to the exceptional pair consisting of the 'last two' members
$E((p-2)\bx_1),E((p-2)\bx_1+\bx_3)$ of the tilting object $T$ is equivalent to $\Tt_{p_1}$, implying that $\Tt_{p_1}$ can be viewed as a triangulated subcategory of $\Tt_p$ for each $p\geq 3$.

This allows the following attempt in to 'define' the notion of an
ADE-chain, by requesting the three properties below:
\begin{enumerate}
\item[(1)] The triangulated categories $(\Tt_n)$ form an infinite
  chain $\Tt_1\subset \Tt_2
  \subset \Tt_2\subset \cdots$ of triangulated categories with Serre
  duality which are fractionally Calabi-Yau;
\item[(2)] Each category $\Tt_n$ has a tilting object $T_n$, hence a
  Grothendieck group which is finitely generated free of rank $n$;
\item[(3)] The sequence formed by the endomorphism rings $A_p=\End{}{T_p}$ is a subsequence
  of a sequence of algebras $(B_n)$ that form an accessible chain of
  finite dimensional algebras in the sense of
  \cite{Lenzing:Pena:2008}, that is, $A_1=k$ and for each integer $n$
  the algebra $A_{n+1}$ is a one-point extension or coextension of
  $A_n$ with an exceptional $A_n$-module.
\end{enumerate}
In our example the request (3) can be satisfied by means of the
algebras $B_n=A(n,3)$. Note, however, that $\Der{\mod{A(n,3)}}$ may fail to
be fractionally Calabi-Yau if $n$ is odd.

\subsection*{The special role of the number 6}

The numbers $6$ and $p-6$ play a special role in dealing with
$\nilop{p}$ and $\snilop{p}$. We advise the reader in this context to
check the paper \cite{Ringel:Schmidmeier:2008a} for the ubiquitous
appearance of the number $6$. Of course this ubiquity of the number
$6$ has its roots in the relationship between $\nilop{p}$ and
$\vect\XX$, where $\XX$ has weight type $(2,3,p)$. The following list
displays a number of appearances of the two numbers:
\begin{enumerate}
\item The group $\LL/\ZZ\vx_3$ is cyclic of order $6$ generated by the
  class of $\vom$.
\item We have $6\vom=(p-6)\vx_3$, thus $\tau^6=\si_3^{p-6}$ holds
  in $\vect\XX$, $\nilop{p}$ and $\snilop{p}$.
\item The partition of line bundles into persistent and fading ones
  obeys the $6$-periodic pattern ${+}\,{-}\,{+}\,{-}\,{-}\,{-}$ in
  each $\tau$-orbit, where $+$ and $-$ stand for persistent and
  fading, respectively.
\item Euler characteristic $\eulerchar\XX$ of $\XX$ and
  fractional Calabi-Yau dimension $d_p$ of $\svect\XX=\snilop{6}$ are
  given by $\chi_\XX=\ffrac{1}{p}-\ffrac{1}{6}$ and
  $d_p=\ffrac{(4p-6)}{3p}$ (up to
  cancelation), respectively.
\item The borderline between (derived) representation-finiteness and
  wildness for $\nilop{p}$ and $\snilop{p}$ is marked by $p=6$.
\item If $p$ and $6$ are coprime, then the Auslander-Reiten
  translation $\tau$ on $\snilop{p}$ has a unique $(p-6)$th root in
  the Picard group.
\end{enumerate}
Additionally we refer to Proposition~\ref{prop:fundamental:domain} for
further occurrences of the number $6$.

\appendix
\section{Existence of a tilting object} \label{sect:appendix}
The existence of a tilting object $T$ is central for many applications discussed in the previous section. We prove the absence of self-extensions of $T$ by a combination of two methods: (a) the determinant argument from Lemma~\ref{lemma_determinant} and (b) the use of line bundle filtrations for the indecomposable summands of $T$. Recall that $\bx_i=\vx_i+\vom$ for $i=1,2,3$.

\begin{lemma} \label{lemma_determinant}
Assume weight type $(2,a,b)$, and let $E$ be an Auslander bundle. Then for any $\vx\in\LL$ we have
$\sHom{}{E}{E(\vx)}=\dual\sHom{}{E}{E(\bx_1-\vx)}$. In particular $\sHom{}{E}{E(\vx)}\neq 0$ implies $2\vx\geq0$ and $2(\bx_1-\vx)\geq 0$.
\end{lemma}
\begin{proof}
  By Serre duality we obtain that
  $\sHom{}{E}{E(\vx)}=\sHom{}{E}{E(\vx-\vx_1)[1]}=\dual\sHom{}{E(\vx-\vx_1)}{E(\vom)}=\dual\sHom{}{E}{E(\bx_1-\vx)}$,
  yielding the first claim. Assume now that $u\colon E\ra E(\vx)$
  becomes non-zero in $\svect{\XX}$. Then Lemma~\ref{lemma:determinant} implies the second claim using that $\determ(E(\vx))-\determ(E)=2\vx$.
\end{proof}
Our next result is fundamental to establish a tilting bundle consisting of Auslander bundles. We provide an elementary proof. For a more conceptual treatment of the topic we refer to~\cite{Kussin:Lenzing:Meltzer:2010pre}.
\begin{proposition} \label{prop:stable_morphisms}
Assume weight type $(2,3,p)$. For each Auslander bundle $E$ we then have
$$
\sHom{}{E}{E(\vx)}\neq0 \quad \iff \quad \vx\in\set{0,\bx_1,\bx_2,\bx_3}.
$$
Moreover, there exist monomorphisms $u_i\colon E\ra E(\bx_i)$ such that
$$
\sHom{}{E}{E(\bx_i)}=\Hom{}{E}{E(\bx_i)}=k\,u_i\quad i=1,2,3.
$$
Further, the morphisms $u_1$, $u_2u_3$ and $u_3u_2$ agree, up to multiplication by a non-zero scalar.
\end{proposition}
\begin{proof}
\underline{Step 1:} We show first that $\sHom{}{E}{E(\bx_i)}=\Hom{}{E}{E(\bx_i)}$. Assume that we have a factorization $[E\up\alpha L(\vx)\up\beta E\bx_i]\neq0$ for some $\vx\in\LL$, and some $i=1,2,3$. Then ($0\leq\vx$ or $\vom\leq\vx$) and ($\vx\leq\bx_i$ or $\vx\leq\bx_i+\vom$). This only leaves the possibilities $\vx=\vom$ or $\vx=\bx_i$. The choice $\vx=\vom$ yields a non-zero member $E\up\alpha L(\vom)\up{j}E$ of $\End{}{E}=k$ resulting in a splitting of $E$. The second choice $\vx=\bx_i$ similarly yields a splitting of $E(\bx_i)$ which again is impossible.

\underline{Step 2:} We know already that $\sEnd{}{E}=k$. Applying $\Hom{}{-}{L(\bx_i)}$ to $\eta\colon 0 \ra L(\vom) \up{\iota}E \up{\pi}L \ra 0$ we obtain an isomorphism

$$(1)\quad \Hom{}{E}{L(\bx_i)}\up{-\circ\iota}\Hom{}{L(\vom)}{L(\bx_i)}.$$
Then applying $\Hom{}{E}{-}$ to $\eta(\bx_i)$, we obtain another isomorphism
$$(2)\quad \Hom{}{E}{E(\bx_i)}\up{\pi(\bx_i)\circ -}\Hom{}{E}{L(\bx_i)}.$$
Combining (1) and (2) we see that $\Hom{}{E}{E(\bx_i)}\iso\Hom{}{L}{L(\vx_i)}=k$. This shows the existence of the $u_i$. Assume that there is a non-zero composition $E\up\alpha L(\vx)\up\beta E(\bx_i)$ for some line bundle $L(\vx)$, then $\vx=\omega$ or $\vx=\bx_i$. In either case there results a non-trivial splitting of $E$ or $E(\bx_i)$ which is impossible. Hence $\sHom{}{E}{E(\vx_i)}=\Hom{}{E}{E(\vx_i)}=ku_i$ for $i=1,2,3$. By Step~1 the $u_i$ are monomorphisms, hence the second claim follows from $\bx_1=\bx_2+\bx_3$.

\underline{Step 3:} We have (a) $\sHom{}{E}{E(\vx)}=0$ for all $\vx>0$, (b) $\sHom{}{E}{E(\bx_2-a\bx_3)}=0$ for all $a\geq1$, and (c) $\sHom{}{E}{E(\bx_3+a\vx_3)}=0$ for all $a\geq1$.

\underline{ad (a):} By Lemma~\ref{lemma_determinant} it is equivalent to show that $\sHom{}{E(\vx)}{E(\bx_1)}=0$. By Proposition~\ref{prop:suspension} we have $\dual{\Hom{}{E(\vx_1)}{E(\bx_1)}}=\sHom{}{E(\vx_1)}{(E(\vx_1))[1]}=0$, we may hence assume that $0<\vx\neq\vx_1$. Now we have $\vx_1-\vx\not\geq0$ and $(\bx_1+\vom)-\vx\geq0$, since otherwise $\bx_1\geq0$, respectively $\vx_1+\vom\geq0$ which both is impossible. Finally $\bx_1-(\vx+\vom)=\vx_1-\vx\not\geq0$ since otherwise $0<\vx <\vx_1$ which again is impossible. We conclude that $\Hom{}{E(\vx)}{E(\bx_1)}=0$ for $\vx>0$.

\underline{ad (b) and (c):} Since $\bx_1-(\bx_3+a\vx_3)=\bx_2-a\vx_3$, it suffices to show property (b). As for the proof of (a) we check the existence of non-zero morphisms between the line bundle factors $L,L(\vom)$ for $E$ and $\bx_2-a\vx_3+\vom$ for $E(\bx_2-a\vx_3)$. That is, for $\vy=\bx_2-a\vx_3$ we need to show that none of $\vy,\vy\pm\vom$ is $\geq0$. Indeed, $\vy=\bx_2-a\vx_3=\vc-\vx_1-(a+1)\vx_3\not\geq0$, $\vy-\vom=\vx_2-a\vx_3\not\geq0$, and $\vy+\vom=\vc-\vx_2-(a+2)\vx_3\not\geq0$.

\underline{Step 4:} We are now in a position to prove the first claim of the proposition. Assume $\vx\notin\set{0,\vx_1,\vx_2,\vx_3}$ and $\sHom{}{E}{E(\vx)}\neq0$. By Step~3 we may assume that $2\vx\not\geq0$, $\bx_1-\vx\not\geq0$, and (*) $0<2\vx<2\bx_1=\vx_2+(p-2)\vx_3$. Therefore the normal form of $\vx$ has the form (**) $\vx=\ell_1\vx_1+\ell_2\vx_2+\ell_3\vx_3-\vc$ where $0\leq\ell_1\leq1$, $0\leq\ell_2\leq2$ and $0\leq\ell_3\leq p-1$. Comparison of (*) and (**) then only allows the following two cases: 1) $2\vx=a\vx_3$ and 2) $2\vx=\vx_2+a\vx_3$ with $0\leq a\leq p-2$.

\underline{Case 1:} Here we obtain $\ell_2=0$, implying that $\ell_1=1$ and $p/2\leq\ell_3\leq p-1$. We then get $\vx = \vx_1 +\ell_3\vx_3 -\vc
    = \bx_2 -a\vx_3$ with $a=p-(\ell_3+1)>0$
since the possibility $\vx=\bx_2$ was excluded. Now Step~3(b) yields $\sHom{}{E}{E(\vx)}=0$, contradicting our assumption.

\underline{Case 2:} Here we obtain $\ell_2=2$, leaving just two possibilities for $(\ell_1,\ell_3)$:

a) $(\ell_1,\ell_3)=(0,\ell_3)$ with $p/2 \leq \ell_3\leq p-1$, and

b) $(\ell_1,\ell_3)=(1,\ell_3)$ with $0\leq \ell_3\leq (p-2)/2$.

In case a) we have $\vx=2\vx_2 +\ell_3\vx_3-\vc$ with $p/2 \leq \ell_3\leq p-1$. Since $\bx_1-\vx=a\vx_3$ with $a=p-(\ell_3+1)\geq0$, then Step~3(a) implies in view of Lemma~\ref{lemma_determinant} that $\sHom{}{E}{E(\vx)}=0$, contradicting our assumption. Finally in case b) we have $\vx=\vx_1+2\vx_2+\ell_3\vx_3-\vc=\vx_1-\vx_2+\ell_3\vx_3=\bx_3+\ell_3\vx_3$ with $\ell_3>0$. Then Step~3(c) implies the contradiction $\sHom{}{E}{E(\vx)}=0$.
\end{proof}

\begin{theorem} \label{thm:tilting}
Assume $\XX$ has weight type $(2,3,p)$, $E$ is an Auslander bundle on $\XX$, and $M=\Set{a\bx_1+b\bx_3}{a=0,\ldots,p-2\textrm{ and } b=0,1}$. Then $T=\bigoplus_{\vx\in M}E(\vx)$ is a tilting object in $\svect\XX$ with endomorphism ring isomorphic to $A(2(p-1),3)$.
\end{theorem}
\begin{proof}
\underline{Endomorphism ring}: To calculate $\sEnd{}{T}$ we arrange the summands $E(\vx)$ according to the scheme
$$
\scriptsize
\xymatrix @-1pc @!R=12pt @!C=12pt{
              &\bx_3\ar[dr]_{u_2}&& \bx_1+\bx_3\ar[dr]_{u_2}&&2\bx_1+\bx_3& \cdots&\ar[dr]_{u_2}&&(p-1)\bx_1+\bx_3\\
0\ar[ur]_{u_3}&&\bx_1\ar[ur]_{u_3}&&2\bx_1\ar[ur]_{u_3}&&&&      (p-1)\bx_1\ar[ur]_{u_3}&\\
}
$$
Then Proposition~\ref{prop:stable_morphisms} implies the claim on the endomorphism ring of $T$.

(A) $T$ is extensionfree:
We show that $\sHom{}{E(\vx)}{E(\vy)[n]}\neq0$ with $\vx,\vy\in M$ and $n\in\ZZ$ implies that $n=0$. From the assumption we obtain the existence of $\vz\in\set{0,\bx_1,\bx_2,\bx_3}$ where $\vc=\vy-\vx+n\vx_1$, and then $\vz=\alpha\bx_1+\beta\bx_3+n\vx_1$ with $|\alpha|\leq p-2$, $|\beta|\leq1$ and $n\in\ZZ$. Passing to congruences in $\LL$ modulo the subgroup $\langle \vx_1,\vx_2\rangle$, generated by $\vx_1$ and $\vx_2$, we obtain
$$
-\alpha\vx_3\equiv\vz\equiv
\begin{cases}
0 & \textrm{if } \vz\in \set{0,\bx_3}\\
-\vx_3& \textrm{if } \vz\in\set{\bx_1,\bx_2}.
\end{cases}
$$
Since $\LL/\langle\vx_1,\vx_2\rangle$ is cyclic of order $p$ and generated by the class of $\vx_3$, we deduce from $|\alpha|\leq p-2$ that $\alpha=0$ for $\vz\in\set{0,\bx_3}$ (case~a) or $\alpha=1$ for $\vz\in\set{\bx_1,\bx_2}$ (case~b).

\underline{case a:} If $\vz=0$, then $\beta\bx_3=-n\vx_1$ with $|\beta|\leq1$. This is only possible for $\beta=0$ which then implies $n=0$. If $\vz=\vx_3$, then $(\beta-1)\bx_3=-n\vx_1$ with $\beta\in\set{-2,-1,0}$. This is only possible for $\beta=0$, again implying $n=0$.

\underline{case b:} Again, we have to deal with two cases: If $\vz=\bx_1$ then $\alpha=1$, and $\beta\bx_3=-n\vx_1$ with $|\beta|\leq1$. As in case~a this implies $n=0$. If $\vz=\bx_2$, then $\alpha=1$, and $(\bx_1-\bx_3)+(\beta-1)\bx_3=\bx_2-n\vx_1$. Since $\bx_1-\bx_3=\bx_2$, we conclude that $n=0$ as in the second part of case~a.

We have thus shown that $T$ has no self-extensions in $\svect\XX$. It then follows from an $\LL$-graded version of Orlov's theorem, compare~\cite{Kussin:Lenzing:Meltzer:2010pre}, that $T$ has the correct number $2(p-1)$ of indecomposable summands to make $T$ tilting in $\svect\XX$.
\end{proof}

\section*{Acknowledgements}

The authors got aware of the subject treated in this paper when
participating in the \emph{ADE chain workshop}, Bielefeld Oct 31 to
Nov 1, 2008, organized by C.~M.\ Ringel. It there came as a surprise
that the categories $\snilop{p}$ and $\svect{\XX(2,3,p)}$ looked to
have a lot of
properties in common; both sequences of triangulated categories
turned out to be good candidates for producing an ADE chain.
The thanks of the authors also go to the critical audience of the Bielefeld
Representation Theory Seminar where the results of this paper were
presented in fall
2009. We thank the referee for a thorough analysis of our paper, and helpful suggestions leading to a significant improvement of the paper. Also we thank S.~Ladkani for pointing out an inaccuracy in a former description of the ADE-chain problem.

The first-named author acknowledges support by the Max Planck Institute for Mathematics, Bonn; the third author was supported by the Polish Scientific Grant Narodowe Centrum Nauki DEC-2011/01/B/ST1/06469.

\bibliographystyle{abbrv}
\def\cprime{$'$}

\end{document}